\documentclass[final]{siamltex}
\usepackage{a4} 
\usepackage{amsmath}
\usepackage{amssymb}
\usepackage{graphicx}
\usepackage{caption,subcaption}
\usepackage{multirow}
\usepackage{xstring}
\usepackage[utf8]{inputenc}
\usepackage{amstext,amsbsy,amsopn,eucal,enumerate}
\usepackage{tikz}
\usepackage{pgfplots}
\usepackage{tabularx}
\usepackage{mathtools}
\usepackage{todonotes}

\def\R{\mathbb{R}}

\newcommand{\expn}{\operatorname{e}}

\newcommand{\vect}{\operatorname{vec}}
\newcommand{\im}{\operatorname{im}}
\newcommand{\cp}{\operatorname{cp}}
\newcommand{\abs}[1]{\left\lvert #1 \right\rvert}

\newcommand{\beq}{\begin{equation}}
\newcommand{\eeq}{\end{equation}}
\newcommand {\mat}      [1] {\left[\begin{array}{#1}}
\newcommand {\rix}          {\end{array}\right]}
\newcommand {\smat}      [1] {\left[\begin{smallmatrix}{#1}}
\newcommand {\srix}          {\end{smallmatrix}\right]}
\newcommand {\s}      [1] {\begin{smallmatrix}{#1}}
\newcommand {\se}          {\end{smallmatrix}}
\newcommand{\trace}{\operatorname{tr}}

\newtheorem{defn}{Definition}[section]
\newtheorem{remark}{Remark}

\newtheorem{lem}[defn]{Lemma}
\newtheorem{prop}[defn]{Proposition} %Proposition entspricht Satz

\newtheorem{thm}[defn]{Theorem}

\title{Low-dimensional approximations of high-dimensional asset price models}

\author{Martin Redmann\thanks{Martin Luther University Halle-Wittenberg, Institute of Mathematics, Theodor-Lieser-Str. 5, 06120 Halle (Saale), Germany (Email: {\tt 
martin.redmann@mathematik.uni-halle.de})}.
\and Christian Bayer\thanks{Weierstraß Institute for Applied Analysis and Stochastics, Mohrenstr. 39, 10117 Berlin, Germany (Email: {\tt christian.bayer@wias-berlin.de})}
\and Pawan Goyal\thanks{Max Planck Institute for Dynamics of Complex Technical Systems, Sandtorstr. 1, 39106 Magdeburg, Germany (Email: {\tt goyalp@mpi-magdeburg.mpg.de})}}

\begin{document}

\maketitle

\begin{abstract}
We consider high-dimensional asset price models that are reduced in their dimension in order to reduce the complexity of the problem or the effect of the curse of dimensionality in the context of option pricing.
We apply model order reduction (MOR) to obtain a reduced system. MOR has been previously studied for asymptotically stable controlled stochastic systems with zero initial conditions.
However, stochastic differential equations modeling price processes 
are uncontrolled, have non-zero initial states and are often unstable. Therefore, we extend MOR schemes and combine ideas of techniques known for deterministic systems. This leads to a method 
providing a good pathwise approximation. After explaining the reduction procedure, the error of the approximation is analyzed and the performance of the algorithm is shown conducting several numerical experiments. Within the 
numerics section, the benefit of the algorithm in the context of option pricing is pointed out.
\end{abstract}



\pagestyle{myheadings}
\thispagestyle{plain}
\markboth{M. REDMANN, C. BAYER, AND P. GOYAL}{LOW-DIMENSIONAL APPROX. OF HIGH-DIMENSIONAL ASSET PRICE MODELS}

%%%%%%%%%%%%%%%%%%%%%%

% \todo[inline,caption={}]{I think that we should consider changing the notation in various
%   ways.
%   \begin{itemize}
%   \item $K$ is usually reserved for strike prices, which is a very strong
%     convention. I suggest to replace it by $\Sigma$, with entries $\rho_{ij}$
%     -- but no strong opinions about the replacements.
%   \item I suggest to replace lower case symbols by upper case symbols when
%     processes are meant.
%   \item The interest rate $\mu$ should be replaced by $r$, or maybe $R$ to
%     avoid clash with the dimension?
%   \end{itemize}
% }

\section{Introduction}

In finance we often encounter high-dimensional models, since the
underlying markets are usually high-dimensional. For instance, in equity, take
all stocks comprising the S \& P 500 index (SPX). Fixed income markets exhibit a
myriad of different relevant interest rates. All these are, obviously, only
small snapshots of even larger markets. Of course, in many situations, we are
only interested in a tiny fraction of these markets, which  can be adequately
modeled by a low-dimensional stochastic process. Moreover, if we are interested in
derivatives on SPX, for example, then we may just model the index itself,
disregarding the fine structure. On the other hand, if we consider a larger
portfolio, this may not be possible without introducing
inconsistencies in the model.

From a numerical perspective, high-dimensional models pose severe
difficulties. Indeed, many traditional computational tools suffer from the
\emph{curse of dimensionality}, which essentially states that the
computational work required to compute the relevant quantity of interest up
to a prescribed error tolerance grows exponentially in the dimension $n$ of
the model. Most methods for discretizing partial differential equations (such as finite element and finite difference methods) suffer from the
curse of dimensionality, as do Fourier based methods. Even many deterministic
sampling methods (i.e., tree methods, quasi Monte Carlo) suffer from the curse of
dimensionality in one way or another.\footnote{It is worth pointing
  out, that the effective dimension for QMC and MC method is, in fact, $n$
  multiplied by the number of time-steps, if time discretization is needed.}
The notable exception is, of course, Monte Carlo simulation.

One way to overcome the numerical burden in high-dimensional models is \emph{Model order reduction} (MOR) \cite{morAnt05, antoulas2020interpolatory, benner2017model}. MOR is a 
technique in numerical analysis in order to construct low dimensional
surrogate models that allow to approximate the quantity of interest with the desired
accuracy. As MOR takes the specific quantity of interest into account, reduced
models for different options will generally be different. (The
specific method introduced later will, however, not depend, e.g., on the
specific strike price.)

To fix ideas, suppose that we are given an $n$ dimensional stochastic
volatility model with asset price processes $S(t) \in \R^n$ and the
corresponding stochastic variance processes $v$ -- which will be
one-dimensional in our numerical examples. Consider an option with
payoff $g(CS)$, where $C \in \R^{p \times n}$ and $g: \R^p \to \R$ possibly
non-linear, where we assume that $p \ll n$. For instance, we have $p = 1$ for basket options. Our goal is to construct a Markov process
$\tilde{x}$ taking values in $\R^{\tilde n}$ -- with $\tilde n \ll n$ -- and a matrix $C_1 \in
\R^{p \times \tilde n}$ such that the processes $CS$ and $C_1 \tilde{x}$ are close in
$L^2$. This usually ensures a good approximation of the payoff, i.e., $g(CS) \approx g(C_1\tilde x)$.
In this paper, we present a general strategy for identifying such
processes $\tilde{x}$. We also provide numerical evidence of successful MOR in
several financial applications, in the sense that relative errors of the order
of $10^{-4}$ are regularly achieved with very small $\tilde n$ even when $n \ge 100$.
It should be noted here that we only propose a MOR technique for the asset process $S$ in
this paper, but not for the variance process $v$. This is due to the generally
non-linear dynamics of the variance process, which would require more
complicated MOR strategies and will be explored in future work.

Before explaining the MOR strategy in detail, some conclusions can
already be made based on the fundamental idea.
First note that MOR should not be confused with \emph{Markovian projection},
see \cite{brunick2013mimicking, gyongy1986mimicking}. The underlying problem
is, of course, that the process $CS$ itself is a natural candidate for a reduced model,
but it generally lacks the Markov property. There is, however, a Markov
process $\hat{x}$ taking values in $\R^p$ such that $CS(t)$ and $\hat{x}(t)$
have the same distribution for every $t$. This means that \emph{European} option
prices based on $\hat{x}$ correspond exactly to the prices in the full
model. The coefficients of $\hat{x}$ are, however, not trivial to
obtain. Nonetheless, there has been continuous interest in the financial
community in applications of Markovian projections, see, for instance,
\cite{hambly2016forward} and \cite{bayer2019implied} for two recent examples.

In contrast, the surrogate model $\tilde{x}$ is often easier to construct than
the Markovian projection $\hat{x}$. Moreover, our construction provides
that $\tilde{x}$ is close to $CS$ on path-space, which directly allows the
application to American option pricing. This comes at the price of being only
an approximation, though. Moreover, MOR may provide good low dimensional surrogate
models even in situations when there is no natural low-dimensional
intermediate process, i.e., when $p \approx n$ as of above.

\begin{remark}
  Generally, the surrogate model $\tilde{x}$ does not have any specific financial
  interpretation. Hence, its only justification is the approximation quality
  with respect to~the quantity of interest.
\end{remark}
Additionally, there are many dimension reduction techniques in the
computational finance literature working at the level of a numerical
approximation rather than the model itself, i.e., in contrast to MOR or
Markovian projections, no lower-dimensional model is ever considered, but
rather the dimensions of certain numerical approximations are reduced. A good,
clarifying example might be the use of sparse grid approximations for
numerical integration or solving partial differential equations. The
underlying observation is, of course, that a one-dimensional grid of size $N$ is
turned into a ``tensor-product'' grid of size $N^n$ in dimension $n$. This
explosion can often be avoided by a careful choice of a \emph{sparse}
``subgrid''. Indeed, under suitable (often quite demanding) regularity conditions,
similar accuracy can be achieved with sparse grids of size
asymptotically proportional to $N \log(N)^n$. We refer to
\cite{bungartzgriebel04} for a general exposition of sparse grid methods, in
particular for solving PDEs, and to \cite{BST18,gerstner} for applications of
sparse grid quadrature methods in finance. If sparse grids are used to
discretize pricing PDEs, for example, then the dimensions of the resulting
system of linear equations is drastically reduced as compared to the tensor
product grid by using available low-dimensional structures in the discretized equation.
In contrast, MOR identifies lower-dimensional effective structures already in
the continuous model, before any discretization, allowing us to use less
sophisticated numerical methods in a low-dimensional surrogate models, in
which all components are actually important.

A somewhat complementary method for solving high-dimensional problems in
finance is based on machine learning, in particular deep neural
networks. These methods are praised for their ability to handle very high
dimensional problems, seemingly breaking the curse of dimensionality. Hence, they
often offer effective alternative computational approaches, even without
explicit dimension reduction. We refer to \cite{BGTW19,HJE18} for two recent
examples of applications in computational finance.

\paragraph{Outline of the paper}

After setting the stage in Section~\ref{stochstabgen}, we introduce techniques
to provide quantitative estimates of ``importance'' of projections of the
state for the dynamics of the process in Section~\ref{sec:reach}. These
quantitative estimates are then used in Section~\ref{subsecprocedure} to
 identify especially efficient reduced models. We continue to provide
error bounds in Section~\ref{h2boundsec}. Numerical experiments are reported
in Section~\ref{numerics}, followed by concluding remarks in
Section~\ref{sec:conclusions}. Some general definitions and technical proofs
are presented in an appendix.

\section{Setting and covariance functions}\label{stochstabgen}

Let $W=\left(W_1, \ldots, W_q\right)^T$ be an $\mathbb R^q$-valued with mean zero Wiener process with covariance matrix $\mathbf K=(k_{ij})$, i.e., $\mathbb E[W(t)W^T(t)]=\mathbf K t$ for $t\in [0, T]$, where $T>0$ is the terminal time.
Suppose that $W$ and all stochastic process appearing in this paper are defined on a filtered probability space  $\left(\Omega, \mathcal F, (\mathcal F_t)_{t\in [0, T]}, 
\mathbb P\right)$\footnote{$(\mathcal F_t)_{t\in [0, T]}$ shall be right continuous and complete.}. In addition, we assume $W$ to be $(\mathcal F_t)_{t\in [0, T]}$-adapted and the 
increments $W(t+h)-W(t)$ to be independent of $\mathcal F_t$ for $t, h\geq
0$. We consider the following large-scale Heston type model:
\begin{subequations}\label{original_system}
\begin{align}\label{stochstatenew}
             dx(t)&=Ax(t)dt+\sum_{i=1}^q \sqrt{v(t)}N_i x(t) dW_i(t),\quad x(0)=x_0=B z,\\ \label{output}
             y(t)&= C x(t), \quad t\in[0, T],
            \end{align}
            \end{subequations}
where $A, N_i\in \mathbb R^{n\times n}$ and $C\in \mathbb R^{p\times n}$. Moreover, the set of initial conditions, in which we are interested, is 
spanned by the columns of a matrix $B\in \mathbb R^{n\times m}$, i.e., there
is a vector $z\in \mathbb R^{m}$ such that $x_0= B z$. This assumption allows to construct a reduced-order system that performs well for several initial states. 
However, there are many financial applications, where only a single $x_0$ is of interest. Then, we have $B=x_0$ and $z=1$. 
The scalar $(\mathcal F_t)_{t\in[0, T]}$-adapted stochastic process $\left(v(t)\right)_{t\in[0, T]}$ is non-negative, $\mathbb P$-a.s. bounded from above 
by a constant $c>0$ and called variance process. The variance process is assumed to be bounded for theoretical considerations below. Practically, boundedness is less relevant.
The state dimension $n$ is assumed to be large and the quantity of interest $y$ is rather low-dimensional, i.e., $p\ll n$.\smallskip

Below, the dependence of the state variable on $x_0$ is sometimes indicated by writing $x(t; x_0)$, $t\in [0, T]$, for the solution to (\ref{stochstatenew}). Furthermore, we write 
$M_1\leq M_2$ for two symmetric matrices $M_1$ and $M_2$ if $M_2-M_1$ is symmetric positive semidefinite. 
In order to identity the important states in system (\ref{original_system}), the covariance function and an upper bound for the covariance will be of interest. Therefore, we formulate the following lemmas.
\begin{lem}\label{lemforstab}
The matrix-valued function $\mathbb{E} \left[x(t; x_0)x^T(t; x_0)\right]$, $t\in [0, T]$, is a solution to 
\begin{align}\label{matrixinequal}
\dot X(t)\leq AX(t)+X(t)A^T+c \sum_{i, j=1}^q N_i X(t) N_j^T k_{ij},\quad X(0)=x_0 x_0^T,
    \end{align}
where $k_{ij}$ is the $ij$th entry of the covariance matrix $\mathbf K$.\end{lem}
\begin{proof}
The proof is given in Appendix \ref{secprooflem2}.
\end{proof}\\
We denote the solution to (\ref{stochstatenew}) by $x_c$ if the process $v$ is replaced by its upper bound $c$ ($v\equiv c$). 
We call (\ref{original_system}) Black Scholes model in case the volatility is constant. The covariance function of $x_c$ can be derived through the identity given in the following lemma.
\begin{lem}\label{lemmatrixeq}
The matrix-valued function $\mathbb{E} \left[x_c(t; x_0)x_c^T(t; x_0)\right]$, $t\in [0, T]$, satisfies 
\begin{align}\label{matrixequal}
{\dot X}_c(t) = AX_c(t)+X_c(t)A^T+c \sum_{i, j=1}^q N_i X_c(t) N_j^T k_{ij},\quad X_c(0)=x_0 x_0^T,
    \end{align}
\end{lem}
\begin{proof}
The statement of this lemma is a special case of \cite[Lemma 2.1]{redmannspa2} 
\end{proof}\\
We now formulate a Gronwall type lemma for matrix differential inequalities involving resolvent positive operators. We refer to Appendix \ref{appendixbla} for a definition of these operators.
\begin{lem}\label{lem3}
Suppose that $L$ is a resolvent positive operator on the space of symmetric matrices. Let the matrix-valued function $X(t)\geq 0$, $t\in [0, T]$, satisfy
\begin{align}\label{matrix_ineq} 
 \dot X(t)\leq L(X(t))
\end{align}
and let $Z(t)\geq 0$, $t\in [0, T]$, be the solution to the matrix differential equation \begin{align}\label{matrix_eq}
 \dot Z(t) = L(Z(t)).
\end{align}
If $X(0)\leq Z(0)$, we have that $X(t)\leq Z(t)$ for all $t\in [0, T]$.
\end{lem}
\begin{proof}
The proof of this theorem for a special resolvent positive operator is given in \cite[Lemma 3.3]{redmann_bil_h2}. 
In order to render this paper as self-contained as possible, the proof is stated in Appendix \ref{secprooflem3} using the same arguments.
\end{proof}\\
Lemma \ref{lem3} together with Lemmas \ref{lemforstab} and \ref{lemmatrixeq} implies that \begin{align}\label{estimateforcov}
   \mathbb{E} \left[x(t; x_0)x^T(t; x_0)\right]\leq \mathbb{E} \left[x_c(t; x_0)x_c^T(t; x_0)\right],                                                                                        
                                                                                          \end{align}
since $L(X):= AX+XA^T+c \sum_{i, j=1}^q N_i X N_j^T k_{ij}$ defines a resolvent positive operator on the space of symmetric matrices, see Appendix \ref{appendixbla}.
This means that the covariance function of a suitable Black Scholes model dominates the one of a Heston model in case the volatility function is bounded.
\begin{remark}
We can use the same approach if we allow for a different volatility $v_i$ ($i=1, \ldots, q$) in every summand of the diffusion in (\ref{stochstatenew}). Then, boundedness has to be understood in a more general sense, i.e., 
we need the existence of a positive semidefinite matrix $\mathbf C =(c_{ij})_{i, j=1, \ldots, q}$ such that \begin{align*}
 \left(v^{\frac{1}{2}}_1(t), \ldots, v^{\frac{1}{2}}_q(t)\right)^T \left(v^{\frac{1}{2}}_1(t), \ldots, v^{\frac{1}{2}}_q(t)\right)\leq \mathbf C                                                                                     
                                                                                      \end{align*}
for all $t\in [0, T]$. The operator $L$ in Lemma \ref{lemforstab} then becomes $L(X)= AX+XA^T+\sum_{i, j=1}^q N_i X N_j^T c_{ij} k_{ij}$. The associated Black Scholes model that guarantees the identity as 
in Lemma \ref{lemmatrixeq}, is given by setting $v_i\equiv 1$ and replacing the Wiener process with covariance matrix $\mathbf K$ by a Wiener process with covariance $\mathbf K\circ \mathbf C$, where $\cdot\circ \cdot$ denotes the 
component-wise product of two matrices. Notice that $\mathbf K\circ \mathbf C$ is positive semidefinite again due to Schur's product theorem \cite{schurprod}.
                                                                                      \end{remark}

\section{Characterization of dominant states}
\label{sec:reach}

We are interested in the dominant subspace of system (\ref{original_system}) meaning that we aim to identify states that are less important in both equations
(\ref{stochstatenew}) and (\ref{output}). Those can be neglected in the system dynamics, leading us to an approximation of the system in a lower dimension.\smallskip

The objects that we choose to identify unimportant states are related to matrices that are used in deterministic control theory. In linear deterministic control systems, the so-called reachability Gramian characterizes 
the minimal energy that is needed to steer a system from zero to some desired state at time $T$. Moreover, the observability Gramian determines the energy that is caused by the observations of an unknown initial state 
on the time interval $[0, T]$, see, e.g., \cite{morAnt05}. Consequently, these Gramians can be used to identify states that require a large amount of energy to be reached and states that produce only very little observation 
energy. Those are less relevant in a control system.\smallskip

We use these ideas and extend them to the stochastic uncontrolled framework considered here. The matrices identifying the dominant subspaces of system (\ref{original_system}) will also be called Gramians due to the link 
between the concepts.

\subsection{Dominant subspaces of (\ref{stochstatenew})}

We introduce the fundamental solution to (\ref{stochstatenew}) as an $\mathbb R^{n\times n}$-valued stochastic process $\Phi$ solving
\begin{align}\label{funddef}
 \Phi(t)=I+\int_0^t A \Phi(s) ds+\sum_{i=1}^q \int_0^t \sqrt{v(s)} N_i \Phi(s)dW_i(s), \quad t\in [0, T],
\end{align}
where $I$ denotes the identity matrix. If $v\equiv c$, the fundamental solution is denoted by $\Phi_c$. 
It is not hard to see that the solution to (\ref{stochstatenew}) is given by \begin{align}\label{solreprexplicit}
    x(t; x_0)= \Phi(t)x_0 = \Phi(t) Bz,                                 
\end{align}
because we assumed that the initial state is spanned by the columns of $B$. Below, $\langle \cdot, \cdot \rangle_2$ denotes the Euclidean inner product and $\left\|\cdot \right\|_2^2$ is the corresponding norm.

Based on (\ref{solreprexplicit}), let us now identify the states in (\ref{stochstatenew}) that play a minor role. We obtain \begin{align}\nonumber
\mathbb E \left\vert\langle x(t; x_0), \tilde x \rangle_2\right\vert^2 &=\mathbb E 
\left\vert \langle \Phi(t) B z, \tilde x\rangle_2 \right\vert^2= \mathbb E \left\vert\langle z, B^T\Phi^T(t) \tilde x \rangle_2 \right\vert^2\\ \label{firstestl2}
&\leq 
\tilde x^T \mathbb E \left[\Phi(t) B B^T \Phi^T(t)\right]\tilde x \left\|z \right\|_2^2 
\end{align}
for a given vector $\tilde x\in \mathbb R^n$ and using Cauchy's inequality. Since $\mathbb E \left[\Phi(t) B B^T \Phi^T(t)\right]$ might not be available
from the computational point of view, we find an estimate based on $\Phi_c$ in the following proposition.
\begin{prop}\label{propestfundsol}
Let $\Phi$ be the fundamental solution to (\ref{stochstatenew}) and suppose that $\Phi_c$ is the fundamental solution to (\ref{stochstatenew}) for the special case $v\equiv c$. Then, we have
\begin{align*}
 \mathbb E \left[\Phi(t)B B^T \Phi^T(t)\right]\leq \mathbb E \left[\Phi_c(t)B B^T \Phi_c^T(t)\right].
\end{align*}
\end{prop}
\begin{proof}
We denote the $i$th column of the matrix $B$ by $b_i$, allowing us to write $\Phi(t)B=\mat{ccc}x(t; b_1),\ldots,x(t; b_m)\rix$. 
Hence, we have\begin{align}\label{repsofF}
\mathbb E \left[\Phi(t)B B^T \Phi^T(t)\right]=\sum_{k=1}^m \mathbb E \left[x(t; b_k)x^T(t; b_k)\right].
\end{align}
Applying (\ref{estimateforcov}) to (\ref{repsofF}) yields \begin{align*}
\mathbb E \left[\Phi(t)B B^T \Phi^T(t)\right]\leq\sum_{k=1}^m \mathbb E \left[x_c(t; b_k)x_c^T(t; b_k)\right] = \mathbb E \left[\Phi_c(t)B B^T \Phi_c^T(t)\right].
\end{align*} 
This concludes the proof.
\end{proof}\\
Combining (\ref{firstestl2}) with Proposition \ref{propestfundsol}, we find \begin{align}\label{interpgram1}
      \mathbb E \left\vert\langle x(t; x_0), \tilde x \rangle_2\right\vert^2\leq \tilde x^T F(t)\tilde x \left\|z \right\|_2^2,
                        \end{align}
where $F(t):=\mathbb E \left[\Phi_c(t) B B^T \Phi_c^T(t)\right]$. We define $P_T:= \int_0^T F(t) dt$ and call $P_T$ (time-limited) reachability Gramian.
Integrating both sides of (\ref{interpgram1}) over $[0, T]$ yields \begin{align}\label{interpgram2}
      \int_0^T\mathbb E \left\vert\langle x(t; x_0), \tilde x \rangle_2\right\vert^2 dt\leq \tilde x^T P_T \tilde x \left\|z \right\|_2^2.
                        \end{align}
Consequently, the Gramian $P_T$ characterizes the relevant subspaces as we see in the next proposition.
\begin{prop}
Let $x(\cdot; x_0)$ be the solution to (\ref{stochstatenew}) with initial state $x_0 =Bz$, i.e., it is spanned by the columns of $B$. Then, it holds that
\begin{align*}
                    x(t; x_0)\in \im(P_T) \quad \mathbb P\otimes dt\text{-a.s. on } \Omega\times[0, T],
                                                                       \end{align*} 
where $\im(\cdot)$ denotes the image of a matrix.
\end{prop}
\begin{proof}
If $\tilde x\in\ker(P_T)$, i.e., $\tilde x$ lies in the kernel of $P_T$, then the left-side of (\ref{interpgram2}) is zero, which implies that $\langle x(t;x_0), \tilde x \rangle_2=0$ $\mathbb P\otimes dt$-a.s.
Since $P_T$ is symmetric positive semidefinite, this yields the claim.
\end{proof}\\
Thus, the states that are not in $\im(P_T)$ are not important in equation (\ref{stochstatenew}). However, it is also important to identify the states that play a minor role. Therefore, we 
turn our attention to states that lie in $\im(P_T)$ but that are nevertheless less important. We can choose an orthonormal basis of eigenvectors $(p_{k})_{k=1,\ldots, n}$ of $P_T$ with associated eigenvalues
 $(\lambda_{k})_{k=1,\ldots, n}$. Then, the following representation \begin{align*}
 x(t; x_0)=\sum_{k=1}^n  \left\langle x(t; x_0), p_{k} \right\rangle_2 p_{k}  \end{align*}
holds. Setting $\tilde x= p_k$ in (\ref{interpgram2}) leads to \begin{align}\label{interpgram21}
      \int_0^T\mathbb E \left\vert\langle x(t; x_0), p_k \rangle_2\right\vert^2 dt\leq \lambda_k \left\|z \right\|_2^2.
                        \end{align}
Consequently, $x(t; x_0)$ is small in the direction of $p_{k}$ if $\lambda_{k}$ is small. Hence, 
states with a large component in the direction of such a $p_{k}$ are less relevant. This means that
 that eigenspaces of $P_T$ corresponding to small eigenvalues $\lambda_k$ play a minor role in the system dynamics. 
\begin{remark}\label{reamrkstable}
$P_T$ is related to the Gramian used in \cite{inhomInitial}. However, they choose $\lim_{T\rightarrow \infty} P_T$ in some asymptotically stable deterministic setting , i.e., 
$N_i=0$ and $\lambda(A)\subset \mathbb C_-$,  where $\lambda(\cdot)$ denotes the spectrum of a matrix. In the stochastic case the respective stability condition were
$\mathbb E\left\|x_c(t;x_0)\right\|_2^2\rightarrow 0$ for $t\rightarrow \infty$ and all initial conditions $x_0$ (mean square asymptotic stability), see, e.g., \cite{damm, staboriginal, redmannspa2}. 
Stability is not assumed in this paper such that $\lim_{T\rightarrow \infty} P_T$ does not exist in general. Moreover, the motivation to use the reachability Gramian $P_T$ is different from the motivation given in \cite{inhomInitial}.
\end{remark}\\                      
We conclude this section by a discussion on how to compute $P_T$ which allows to identify redundant information in the system. Using the representation of $F$ in (\ref{repsofF}) and applying 
Lemma \ref{lemmatrixeq} to every summand of its right-side, we see that $F$ satisfies 
\begin{align}\label{matrixequalforF}
{\dot F}(t) = AF(t)+F(t)A^T+c \sum_{i, j=1}^q N_i F(t) N_j^T k_{ij},\quad F(0)=BB^T.
    \end{align}
Integrating both sides of (\ref{matrixequalforF}) yields \begin{align}\label{comPT}
F(T)-BB^T = AP_T+P_T A^T+c \sum_{i, j=1}^q N_i P_T N_j^T k_{ij}.
    \end{align}
This means that the large-scale generalized Lyapunov equation (\ref{comPT}) needs to be solved to derive $P_T$. This can be done also in a large-scale setting for a given left-side.
However, the left-side of (\ref{comPT}) depends on $F(T)$, a matrix that 
needs to be computed beforehand. For dimensions $n$ of a few hundreds, this can be done directly by vectorizing (\ref{matrixequalforF}). Defining $f(t):=\vect(F(t))$, we then obtain 
\begin{align}\label{matrixequalforFvec}
\dot{f}(t) = \mathcal K f(t),\quad f(0)=\vect(BB^T),
    \end{align}
where $\vect(\cdot)$ is the vectorization of a matrix, $\cdot \otimes \cdot$ denotes the Kronecker product of two matrices and \begin{align*}
                                                                                                       \mathcal K:= I\otimes A + A\otimes I+ c \sum_{i, j=1}^q N_i \otimes N_j k_{ij}.
                                                                                                                               \end{align*}
Consequently, deriving $F(T)$ relies on the efficient computation of a matrix exponential, since
\begin{align}\label{comptFT}
 f(T) = \expn^{\mathcal K T}\vect(BB^T).
\end{align}
A discussion on how to determine a matrix exponential efficiently can be found in \cite{morK17} and references therein.
However, we need to assume that $0\not\in \lambda (\mathcal K)$. This guarantees a unique solution of (\ref{comPT}) which we suppose to have 
below.  If system \eqref{original_system} were mean square asymptotically stable in the spirit of Remark \ref{reamrkstable}, then we could take $T\rightarrow \infty$ in (\ref{comPT}) and $F(T)$ would disappear in the limit 
which makes the computation of the (infinite) reachability Gramian much simpler. Such type of (infinite) reachability Gramians are, e.g., 
used to characterize reachability energies in mean square asymptotically stable controlled stochastic systems having time-invariant coefficients \cite{redmannbenner, redmannspa2}, a setting that differs significantly 
from the one considered here.\smallskip

More advanced approaches need to be used to solve for $F(T)$ if $n$ is very large. There are relevant examples in which $F(T)$ and hence $P_T$ can be derived explicitly 
as we will see in Section \ref{numerics}.  

We give some insight on the computational complexity of the solution to the Lyapunov equation in the following remark.

\begin{remark}
Once we have determined $F(T)$, a large scale generalized Lyapunov equation (\ref{comPT}) needs to be solved. If one uses vectorization and the Kronecker product to solve the matrix equation, which then becomes
a linear system of size $n^2$, then in general $\mathcal{O}(n^6)$ operations are needed to solve this linear system.\smallskip

However, in the past decades, several techniques have been developed in order to solve the large scale generalized Lyapunov equation (\ref{comPT}) more efficiently.
Damm \cite{damm_lyap_eq} has shown that one can compute the solution to (\ref{comPT}) by solving a sequence of standard Lyapunov equations, that is the case where $N_i=0$ for $i=1, \ldots, q$. Such standard Lyapunov equations can either
be solved by direct methods, such as Bartels-Stewart \cite{Bartels_Stewart}, which cost $\mathcal{O}(n^3)$ operations, or by iterative methods such as ADI or Krylov subspace methods
\cite{Simoncini_ueberblick}, which have a much  smaller complexity than the Bartels-Stewart algorithm, in particular, when the left hand side is of low rank or structured. 
In addition, several low rank solvers have been developed for computing the solution to (\ref{comPT}) directly, for the case where one can show that the solution is approximately
of low rank \cite{benner2013low,kressner2015truncated,shank2016efficient}. The complexity of those methods is generally $\mathcal{O}(n^2)$ or less.
\end{remark}

\subsection{Dominant subspace of (\ref{output})}\label{diffobs}

 We now characterize the importance of an initial state $x_0$ in the output equation. The initial condition is not relevant if the corresponding output $y(\cdot; x_0)$ has zero energy and is of low relevance if the 
 output energy is small, since those initial states barely contribute to the quantity of interest. We begin with an estimate for $y$ based on the result of Section \ref{stochstabgen}. 
 \begin{prop}
 Suppose that $y$ is given by (\ref{output}) and let us assume that $y_c$ is the output associated with the solution to the solution of (\ref{stochstatenew}) if $v\equiv c$, i.e., $y_c(t)=C x_c(t)$. Then, we have 
 \begin{align}\label{boundout}
\mathbb E\int_0^T \left\|y(t) \right\|_2^2dt \leq \mathbb E\int_0^T \left\|y_c(t) \right\|_2^2dt.
 \end{align}
 \end{prop}
\begin{proof}
We use the linearity of the trace to obtain
  \begin{align*}
\mathbb E\int_0^T \left\|y(t) \right\|_2^2dt = \mathbb E\int_0^T \trace(Cx(t)x^T(t)C^T ) dt =\int_0^T \trace(C\mathbb E[x(t)x^T(t)]C^T ) dt.
 \end{align*}
Using that (\ref{estimateforcov}) is preserved when the trace is applied yields 
 \begin{align*}
\mathbb E\int_0^T \left\|y(t) \right\|_2^2dt \leq \int_0^T \trace(C\mathbb E[x_c(t)x_c^T(t)]C^T ) dt = \mathbb E\int_0^T \left\|y_c(t) \right\|_2^2dt
 \end{align*}
 which concludes the proof of this proposition.
 \end{proof}\\
 Now, the goal is to find a bound for the energy of $y_c$. Therefore, we introduce $Q_T$ as the solution to
\begin{align}\label{gengenlyapobs}
G(T)-C^T C = A^T Q_T+Q_T A+c\sum_{i, j=1}^q N_i^T Q_T N_j k_{ij}, 
\end{align}
an equation that can be solved for large $n$ once the left-side is given. We refer to $Q_T$ as the observability Gramian since it characterizes the observation 
energy as we will see below. $G(t)$, $t\in [0, T]$, entering in (\ref{gengenlyapobs}) satisfies \begin{align}\label{ODEforG}
\dot G(t)= A^T G(t)+G(t) A+c\sum_{i, j=1}^q N_i^T G(t) N_j k_{ij}, \quad G(0)=C^T C,
\end{align}
i.e., $Q_T= \int_0^T G(s) ds$. Notice that if $n$ is not too large, $G(T)$ can be computed analogously to (\ref{comptFT}) meaning that $g(T):=\vect(G(T))$ is given by
\begin{align*}
 g(T) = \expn^{\mathcal K^T T}\vect(C^TC).
\end{align*}
Below, we distinguish between two cases. We first discuss the case in which the system matrices commute.
\paragraph{Commuting matrices} We find a representation for $Q_T$ and subsequently an energy estimate for $y_c$ in case all the matrices in (\ref{stochstatenew}) commute. For that purpose, we establish the following
result. \begin{prop}\label{propcommute}
Let us assume that all matrices $A, N_1, \ldots, N_q$ commute. Hence, we have that these matrices commute with the fundamental solution $\Phi_c$, i.e.,
\begin{align}\label{matric_commute}
A \Phi_c(t) = \Phi_c(t) A\quad \text{and}\quad N_i \Phi_c(t) = \Phi_c(t) N_i
\end{align}
for all $t\in[0, T]$ and $i=1, \ldots, q$.
        \end{prop}
\begin{proof}
These identities hold since the left and the right-sides satisfy the same differential equation, e.g., one can multiply (\ref{funddef}) with $A$ from the left to obtain the 
equation for $A \Phi_c$ and with $A$ from the right to get the one for $\Phi_c A$. Since all system matrices commute, the equations coincide. Similarly, one finds the indents for the matrices $N_i$.
\end{proof}\\
The example considered in Section \ref{numerics} satisfies the assumption of Proposition \ref{propcommute}. Furthermore, notice that 
in the deterministic case ($N_i=0$), we have $\Phi_c(t)=\expn^{At}$ and hence (\ref{matric_commute}) is always given. Based on Proposition \ref{propcommute}, a representation of 
$Q_T$ can be found.
\begin{prop}\label{QTifcommute}
Under the assumptions of Proposition \ref{propcommute}, we have \begin{align*}
                                                              Q_T= \int_0^T \mathbb E\left[\Phi_c^T(t) C^T C \Phi_c(t)\right] dt.
                                                                \end{align*}
\end{prop}
\begin{proof}
We apply Ito's product rule to $\Phi_c^T(t) C^T C \Phi_c(t)$ and take the correlation of the noise processes into account. This yields
\begin{align*}
&d\left(\Phi_c^T(t) C^T C \Phi_c(t)\right) \\
&= d\left(\Phi_c^T(t)\right) C^T C \Phi_c(t) +  \Phi_c^T(t) C^T C d\left(\Phi_c(t)\right) +  d\left(\Phi_c^T(t)\right) C^T C d\left(\Phi_c(t)\right)\\
&= d\left(\Phi_c^T(t)\right) C^T C \Phi_c(t) +  \Phi_c^T(t) C^T C d\left(\Phi_c(t)\right) +   c \sum_{i, j=1}^q \Phi_c^T(t) N_i^T C^T C N_j  k_{ij} \Phi_c(t) dt.
\end{align*}
Above, we plug in (\ref{funddef}) for the case when $v\equiv c$ and take the expected value on both sides. Hence, using that the Ito integral has mean zero,  we have
\begin{align*}
&d\left(\mathbb E\left[\Phi_c^T(t) C^T C \Phi_c(t)\right]\right) \\
&= \mathbb E\left[\Phi_c^T(t)\left( A^T C^T C + C^T C A +  c\sum_{i, j=1}^q N_i^T C^T C N_j  k_{ij} \right)\Phi_c(t)\right] dt.
\end{align*}
Due to (\ref{matric_commute}) we see that $\mathbb E\left[\Phi_c^T(t) C^T C \Phi_c(t)\right]$, $t\in [0, T]$, solves (\ref{ODEforG}) and thus $Q_T= \int_0^T \mathbb E\left[\Phi_c^T(t) C^T C \Phi_c(t)\right] dt$.
\end{proof}\\
Inequality (\ref{boundout}) and Proposition \ref{QTifcommute} now imply that \begin{align}\label{outencom}
\mathbb E\int_0^T \left\|y(t) \right\|_2^2dt \leq \mathbb E\int_0^T \left\|C \Phi_c(t) x_0 \right\|_2^2dt = x_0^T Q_T x_0.
 \end{align}
Initial states that are spanned by eigenvectors of $Q_T$ belonging to the small eigenvalues lead to a small right-side in (\ref{outencom}) and consequently yield a small output $y$.
Hence, we know that eigenspaces of $Q_T$ corresponding to the small eigenvalues are less relevant in (\ref{output}). 
\paragraph{General case}
We find another bound on the energy of $y_c$ and hence also for $y$ in the general case.
\begin{prop}
If $y$ is the quantity of interest in system (\ref{original_system}) and $Q_T$ the solution of (\ref{gengenlyapobs}), then 
 \begin{align}\label{outenergiewithremainder}
\mathbb{E} \int_0^T\left\| y(t)\right\|_2^2 dt \leq x_0^T Q_T x_0 +\mathcal R(T),   
                             \end{align}
where $\mathcal R(T):= \mathbb{E} \int_0^T x^T_c(t; x_0) G(T) x_c(t; x_0) dt - \mathbb{E} \left[x_c^T(T; x_0) Q_T x_c(T; x_0)\right]$ with $G$ solving (\ref{ODEforG}).
\end{prop}
\begin{proof}
We make use of \begin{align}\label{relationwtrace}
\mathbb{E} \left[x_c^T(T; x_0) Q_T x_c(T; x_0)\right] = \trace(Q_T\mathbb{E} \left[x_c(T; x_0)x_c^T(T; x_0)\right]).                    
                            \end{align}
We obtain from Lemma \ref{lemmatrixeq} that\begin{align*}
            Q_T\mathbb{E} \left[x_c(T; x_0)x_c^T(T; x_0)\right]=\; &Q_T x_0 x_0^T + \mathbb{E} \int_0^T  Q_TA x_c(t; x_0)x_c^T(t; x_0) dt\\
            &+\mathbb E\int_0^T Q_T x_c(t; x_0)x_c^T(t; x_0)A^T dt\\
            &+c \sum_{i, j=1}^q\mathbb E\int_0^T  Q_T N_i x_c(t; x_0)x_c^T(t; x_0) N_j^T k_{ij} dt.
                                       \end{align*}
Using properties of the trace, (\ref{relationwtrace}) becomes
\begin{align*}
&\mathbb{E} \left[x_c^T(T; x_0) Q_T x_c(T; x_0)\right] \\&= x_0^T Q_T x_0+\mathbb{E} \int_0^T x_c^T(t; x_0) \left(A^T Q_T+Q_T A+c \sum_{i, j=1}^q N_i^T Q_T N_j q_{ij} \right) x_c(t; x_0) ds.                
                            \end{align*}
 We insert equation (\ref{gengenlyapobs}) into the above identity to get \begin{align}\nonumber
&\mathbb{E} \left[x_c^T(T; x_0) Q_T x_c(T; x_0)\right] = x_0^T Q_T x_0+\mathbb{E} \int_0^T x^T_c(t; x_0) (G(T)-C^T C) x_c(t; x_0) dt   \\ \label{soobsgram}    
 &= x_0^T Q_T x_0-\mathbb{E} \int_0^T\left\| y_c(t)\right\|_2^2 dt+\mathbb{E} \int_0^T x^T_c(t; x_0) G(T) x_c(t; x_0) dt
                             \end{align}
This, together with (\ref{boundout}), gives us the result.
 \end{proof}\\
 Assuming that the remainder term $\mathcal R(T)$ is not too large, the same conclusions as below 
 (\ref{outencom}) can be made. The eigenspaces that belong to the small eigenvalues of $Q_T$ are unimportant. If the system is mean square asymptotic stable, then 
 $Q_\infty:=\lim_{T\rightarrow \infty}Q_T$ exists and $\mathcal R(T)\rightarrow 0$ as $T\rightarrow \infty$. Taking the limit of $T\rightarrow \infty$ in (\ref{outenergiewithremainder})
 would then lead to a characterization of the output energy  by $Q_\infty$ without a remainder term. $Q_\infty$ is also easier to determine than $Q_T$ since it solves (\ref{gengenlyapobs}) with $G(T)=0$.
 Energy estimates based on $Q_\infty$ are shown in \cite{morBenD11,redmannbenner, redmannspa2} if the variance $v$ is constant.\smallskip
 
We can also get to a more explicit bound by applying Gronwall's lemma to (\ref{soobsgram}) if 
 $Q_T$ is regular. Defining $\alpha(T) := x_0^T Q_T x_0-\mathbb{E} \int_0^T\left\| y_c(t)\right\|_2^2 dt$ (\ref{soobsgram}) becomes  \begin{align*}
\mathbb{E} \left[\left\|Q_T^{\frac{1}{2}} x_c(t; x_0) \right\|_2^2\right] &= \alpha(T)+\mathbb{E} \int_0^T \left\|G^{\frac{1}{2}}(T) x_c(t; x_0) \right\|_2^2 dt\\
&\leq \alpha(T)+ k_T \mathbb{E} \int_0^T \left\|Q_T^{\frac{1}{2}} x_c(t; x_0) \right\|_2^2 dt
\end{align*}
where $k_T:= \left\|G^{\frac{1}{2}}(T)Q_T^{-\frac{1}{2}}\right\|_2^2$. Gronwall's lemma leads to \begin{align*}
0\leq\mathbb{E} \left[\left\|Q_T^{\frac{1}{2}} x_c(t; x_0) \right\|_2^2\right]\leq \alpha(T)+\mathbb{E} \int_0^T \alpha(t) k_T \expn^{k_T (T-t)}  dt.
\end{align*}
With a few more steps, we find \begin{align*}
\mathbb{E} \int_0^T\left\| y(t)\right\|_2^2 dt \leq x_0^T Q_T x_0 \expn^{k_T T},   
                             \end{align*}
but this bound cannot be expected to be tight.
                             
\section{State-space transformation and reduced-order model}\label{subsecprocedure}

Balancing related MOR like balanced truncation were initially invented for controlled linear deterministic systems that are asymptotically stable and have zero initial states \cite{morMoo81}. 
Balanced truncation has been extended to stochastic systems with similar properties \cite{morBenD11, redmannbenner}. Subsequently, this scheme was studied for deterministic and stochastic systems with non-zero initial 
conditions \cite{inhomInitial, stochInhom}. However, all these methods are restricted to stable systems. A method for deterministic equations called time-limited balanced truncation aiming to create a good reduced system on 
a finite time interval only was introduced in \cite{morGawJ90}. As pointed out in \cite{morK17}, this method has some potential in the context of unstable systems.\smallskip

The method explained below is a combination of all the methods mentioned above. It follows the same concept which is simultaneously diagonalizing system Gramians. Here, the Gramians are $P_T$ and $Q_T$ solving 
(\ref{comPT}) and (\ref{gengenlyapobs}), respectively.
We have shown the relevance of these Gramians in Section \ref{sec:reach}. Diagonalizing both $P_T$ and $Q_T$ means that we create a system in which the important states in equations (\ref{stochstatenew}) and 
(\ref{output}) are the same. Hence, the unimportant ones can be easily identified and thus truncated.\smallskip

%The unimportant ones can then be easily truncated.\smallskip

Let $S\in\mathbb{R}^{n\times n}$ be a regular matrix. We do a coordinate transformation by introducing
\[\hat{x}(t) = Sx(t).\]
Based on (\ref{original_system}) the associated system is
\begin{subequations}\label{balancingtransformation}
\begin{align}\label{balstate}
d\hat{x}(t)&=\hat{A} \hat{x}(t)dt+\sum_{i=1}^q\sqrt{v(t)} \hat N_i x(t) dW_i(t),\quad \hat x(0)= Sx_0 = SBz,\\ \label{balout}
y(t)&= \hat{C} \hat{x}(t),\quad t\in [0, T],
\end{align}
\end{subequations}
where $\hat{A} = SAS^{-1}$, $\hat{B}=SB$, $\hat{C} =CS^{-1}$ and $\hat{N_i} = SN_i S^{-1}$. Notice that the quantity of interest does not change with this transformation.
However, the matrices characterizing the importance of states in (\ref{balstate}) and (\ref{balout}) are different ones. For the transformed system 
(\ref{balancingtransformation}), these become \begin{align*}
\hat{P}_T =S P_T S^T\quad \text{and}\quad \hat{Q}_T = S^{-T} Q_T S^{-1}.
\end{align*}
The above relation is obtained by multiplying (\ref{matrixequalforF}) with $S$ from the left and with $S^T$. Moreover, (\ref{ODEforG}) needs to be multiplied with $S^{-T}$ from the left and with 
$S^{-1}$ from the right.\smallskip

We now choose $S$ such that $\hat{P}_T = \hat{Q}_T = \Sigma_T = \diag(\sigma_{1},\ldots,\sigma_n)$, where $\sigma_1\ge\ldots\ge\sigma_n> 0$ are called Hankel singular values (HSVs) 
and given by $\sigma_i =\sqrt{\lambda_i(P_TQ_T)}$, where $\lambda_i(\cdot)$ denotes the $i$th eigenvalue of the matrix and $i=1,\ldots,n$. Such a system is called balanced. A transformation like this 
always exists if $P_T, Q_T>0$. It is, together with its inverse, derived the following way: 
\[
S=\Sigma_T^{-\frac{1}{2}}U^T L_Q^T \quad\text{and}\quad S^{-1}=K_PV\Sigma_T^{-\frac{1}{2}}.
\]
The above matrices are computed from factorizations $P_T = K_PK_P^T$ and $Q_T=L_QL_Q^T$ as well as from the singular value decomposition of $K_P^T L_Q = V\Sigma U^T$. \smallskip

In a balanced system, it is easy to identify the unimportant states. They are the ones corresponding to the small HSVs of the system and represented by $x_2$ given by the partition of the balanced state variable 
\begin{align*}
 \hat x(t) = S x(t) = \smat x_1(t)\\ x_2(t)  \srix,
\end{align*}
where $x_1(t)\in \mathbb R^{\tilde n}$ represents the relevant states in the system dynamics. Furthermore, we partition the balanced realization as follows:
\begin{align*}
S{A}S^{-1}= \smat{A}_{11}&{A}_{12}\\ 
{A}_{21}&{A}_{22}\srix,\quad S{B} = \smat{B}_1\\ {B}_2\srix,\quad {CS^{-1}} = \smat{C}_1 &{C}_2\srix,\quad S{N_i}S^{-1}= \smat{N}_{i, 11}&{N}_{i, 12}\\ 
{N}_{i, 21}&{N}_{i, 22}\srix, \end{align*}
where ${A}_{11}\in\R^{\tilde n\times \tilde n}$ etc. With this, system (\ref{balancingtransformation}) becomes
\begin{align}\label{statpartsys}
\smat d x_1\\ d x_2  \srix&=\smat{A}_{11}&{A}_{12}\\ 
{A}_{21}&{A}_{22}\srix \smat x_1\\ x_2\srix dt+\sum_{i=1}^q\sqrt{v} \smat{N}_{i, 11}&{N}_{i, 12}\\ 
{N}_{i, 21}&{N}_{i, 22}\srix \smat x_1\\ x_2\srix dW_i,\quad \hat x(0)= \smat{B}_1\\ {B}_2\srix z,\\ \label{outtranssys}
y(t)&= \smat{C}_1 &{C}_2\srix \smat x_1(t)\\ x_2(t)\srix,\quad t\in [0, T].
\end{align}
The time dependence is omitted in (\ref{statpartsys}) to shorten the notation.
The reduced system of dimension $\tilde n \ll n$ is now obtained by neglecting $x_2$, i.e., the second line in (\ref{statpartsys}) is truncated the remaining $x_2$ variables are set zero in both the first line of 
(\ref{statpartsys}) and in (\ref{outtranssys}). The reduced-order model then is 
\begin{subequations}\label{generalreducedsys}
\begin{align}\label{romstate}
d\tilde x(t)&=A_{11} \tilde x(t)dt+\sum_{i=1}^q \sqrt{v(t)} {N}_{i, 11} \tilde x(t)dW_i(t),\quad \tilde x(0)=B_1 z,\\ \label{outrom}
\tilde y(t)&= C_1 \tilde x(t),\quad t\in [0, T].
\end{align}
\end{subequations}
where $A_{11}, N_{i, 11}\in\mathbb R^{\tilde n \times \tilde n}$, $B_1 \in\mathbb R^{\tilde n\times m}$ and $C_1\in\mathbb R^{p\times \tilde n}$.

\section{Error bound analysis}\label{h2boundsec}

We introduce an error system by combining (\ref{stochstatenew}) and (\ref{romstate}) with an output equation that represents the error between (\ref{output}) and (\ref{outrom}). The error system is
 \begin{equation}\label{errorsys:original}
\begin{aligned}
d x^e (t) &= A^e x^e(t) dt + \sum_{i=1}^q \sqrt{v(t)} N_i^e x^e(t) dW_i(t),\quad x^e(0) = B^e z,\\
y^e(t) &= C^e x^e(t),\quad t\in [0, T],
\end{aligned}
\end{equation}
where the error state $x^e$ and the error matrices $(A^e, B^e, C^e, N_i^e)$ are
\begin{align}\label{parterrormatr}
x^e =\smat x \\ \tilde x \srix,\; A^e = \smat{A}& 0\\ 
0 &A_{11}\srix,\;B^e=\smat B \\  B_1\srix,\;C^e= \smat C & -C_1 \srix,\; N_i^e=\smat {N}_i& 0 \\ 
0 &N_{i, 11}\srix.
            \end{align} 
Let us again assume that an index $c$ indicates that $v$ in (\ref{errorsys:original}) is replaced by $c$. We obtain 
 \begin{align*}
\mathbb E\int_0^T \left\|y^e(t) \right\|_2^2dt \leq \mathbb E\int_0^T \left\|y^e_c(t) \right\|_2^2dt
 \end{align*}
the same way as in (\ref{boundout}). Based on the fundamental solution $\Phi^e$ (or $\Phi^e_c$), the quantity of interest given a constant volatility function, is represented by $y^e_c(t) = C^e\Phi_c^e(t) B^e z$. Plugging this into the
above inequality yields 
\begin{align}\nonumber
\mathbb E\int_0^T \left\|y^e(t) \right\|_2^2dt \leq &\mathbb E\int_0^T \left\|C^e\Phi_c^e(t) B^e z\right\|_2^2dt \leq \mathbb E\int_0^T \left\|C^e\Phi_c^e(t) B^e\right\|_F^2dt \left\|z\right\|_2^2\\
&=\trace(C^e P_T^e (C^e)^T) \left\|z\right\|_2^2,\label{firsterror}
 \end{align}
where we set $P_T^e:=\int_0^T F^e(t) dt$ with $F^e(t)= \Phi_c^e(t) B^e (B^e)^T (\Phi_c^e)^T(t)$. Analogue to (\ref{matrixequalforF}), $F^e$ solves 
\begin{align}\label{matrixequalforFerrorsys}
{\dot F}^e(t) = A^eF^e(t)+F^e(t)(A^e)^T+c \sum_{i, j=1}^q N^e_i F^e(t) (N_j^e)^T k_{ij},\quad F^e(0)=B^e(B^e)^T.
    \end{align}
We partition the solution to (\ref{matrixequalforFerrorsys}) as follows
\begin{align}\label{decomF}
F^e(t) = \smat{F_{11}(t)}& F_{12}(t)\\ F_{12}^T(t)&{F_{22}(t)}\srix
            \end{align} 
and see that $F_{11}(t)= F(t)$ solving (\ref{matrixequalforF}) as well as $F_{12}(t)= \bar F(t)$ and $F_{22}(t)= \tilde F(t)$ that are the solutions to
\begin{align}\label{matrixequalforFred}
\dot {\bar F}(t) &= A\bar F(t)+\bar F(t)A_{11}^T+c \sum_{i, j=1}^q N_i \bar F(t) N_{j, 11}^T k_{ij},\quad \bar F(0)=BB_1^T\\ \label{matrixequalforFmixed}
\dot {\tilde F}(t) &= A_{11}\tilde F(t)+\tilde F(t)A_{11}^T+c \sum_{i, j=1}^q N_{i, 11} \tilde F(t) N_{j, 11}^T k_{ij},\quad \tilde F(0)=B_1B_1^T
    \end{align}
using the partitions in (\ref{parterrormatr}). From (\ref{firsterror}) and (\ref{decomF}), we obtain
\begin{align}\nonumber
&\mathbb E\int_0^T \left\|y(t)-\tilde y(t) \right\|_2^2dt=\mathbb E\int_0^T \left\|y^e(t) \right\|_2^2dt \\ \label{absoutbound}
&\leq \left(\trace(C P_T C^T)-2 \trace(C \bar P_T C_1^T)+\trace(C_1 \tilde P_T C_1^T) \right)\left\|z\right\|_2^2, \end{align}
where $\bar P_T:=\int_0^T \bar F(t)dt$ and $\tilde P_T:=\int_0^T \tilde F(t)dt$. By Integrating both (\ref{matrixequalforFred}) and (\ref{matrixequalforFmixed}), the equations for these two matrices are 
\begin{align}\label{gramequalforFred}
\bar F(T)-BB_1^T &= A\bar P_T+\bar P_T A_{11}^T+c \sum_{i, j=1}^q N_i \bar P_T N_{j, 11}^T k_{ij},\\ \label{gramequalforFmixed}
\tilde F(T)-B_1B_1^T &= A_{11}\tilde P_T+\tilde P_T A_{11}^T+c \sum_{i, j=1}^q N_{i, 11} \tilde P_T N_{j, 11}^T k_{ij}.
    \end{align}
Since $P_T$ is already known from the balancing procedure explained in Section \ref{subsecprocedure}, the bound for the absolute output error in (\ref{absoutbound}) requires only the computation of $\bar P_T$ and 
$\tilde P_T$. Since the reduced dimension $\tilde n$ is rather small, the corresponding equations (\ref{gramequalforFred}) and (\ref{gramequalforFmixed}) can often be solved directly through vectorization. Hence, we have
\begin{align*}
 \vect(\bar F(T))-\vect(BB_1^T) &= \bar{\mathcal K} \vect(\bar P_T),\\ 
\vect(\tilde F(T))-\vect(B_1B_1^T) &= \tilde{\mathcal K} \vect(\tilde P_T)
\end{align*}
with $\bar F(T)=\expn^{\bar{\mathcal K}T}\vect(BB_1^T)$ and $\tilde F(T)=\expn^{\tilde{\mathcal K}T}\vect(B_1B_1^T)$, where 
\begin{align*}
 \bar{\mathcal K}&:=\left(I_n\otimes A_{11} + A\otimes I_{\tilde n}+ c \sum_{i, j=1}^q N_i \otimes N_{j, 11} k_{ij}\right),\\
  \tilde{\mathcal K}&:=\left(I_{\tilde n}\otimes A_{11} + A_{11}\otimes I_{\tilde n}+ c \sum_{i, j=1}^q N_{i, 11} \otimes N_{j, 11} k_{ij}\right).
\end{align*}
Above, the identity matrices are equipped with an index indicating the respective dimension. With (\ref{absoutbound}) a bound
for the absolute error of reducing system (\ref{original_system}) was found. However, the relative error is more interesting to be 
analyzed. Therefore, we need a computable lower bound for the $L^2$-norm of $y$. This task is relatively simple because the inequality of Cauchy Schwartz yields \begin{align*}
    \int_0^T \left\|\mathbb E[y(t)]\right\|_2^2dt   \leq   \mathbb E\int_0^T \left\|y(t)\right\|_2^2dt.                                                                                                                                         
                                                                                                                                                      \end{align*}
Now, $\mathbb E[x(t)]$, $t\in [0, T]$, solves equation (\ref{stochstatenew}) with $c=0$ which can be seen easily by applying the expected value to both sides of (\ref{stochstatenew}) and by exploiting that the Ito integrals
have zero mean. We use that $\Phi_{c=0}(t)=\expn^{A t}$ such that $\mathbb E[y(t)]= C \expn^{A t} B z$. We plug this into the above estimate leading to \begin{align}\label{lowerboundy}
 \mathbb E\int_0^T \left\|y(t)\right\|_2^2dt\geq  \int_0^T \left\|C \expn^{A t} B z\right\|_2^2dt = z^T B^T Q_{T, 0} B z,   .                                                                                                                                         
                                                                                                                                                      \end{align}
where $Q_{T, 0}:= \int_0^T \expn^{A^T t} C^T C \expn^{A t} dt$. According to Subsection \ref{diffobs}, $Q_{T, 0}$ solves (\ref{gengenlyapobs}) with $c=0$, i.e., 
\begin{align}\label{gengenlyapobsdet}
\expn^{A^T T} C^T C \expn^{A T}- C^T C = A^T Q_{T, 0}+Q_{T, 0} A, 
\end{align}
an equation that can be solved in a large-scale setting, since there are efficient methods to determine $\expn^{A T}$ for large $n$. We summarize the results of this section in a theorem below, where 
we set $\left\|y\right\|_{L^2}^2:= \mathbb E\int_0^T \left\|y(t)\right\|_2^2dt$.
\begin{thm}\label{thmerrorbound}
Let $y$ be the output of the original system (\ref{original_system}) and let $\tilde y$ be the output of the reduced model (\ref{generalreducedsys}). Then, the relative $L^2$-error between $y$ and $\tilde y$ is bounded as follows:
\begin{align}\label{ebnumerics}
\frac{ \left\|y-\tilde y\right\|_{L^2} }{\left\|y\right\|_{L^2}}\leq \frac{\left(\trace(C P_T C^T)-2 \trace(C \bar P_T C_1^T)+\trace(C_1 \tilde P_T C_1^T)\right)^{\frac{1}{2}} \left\|z\right\|_2}{\left(z^T B^T Q_{T, 0} B z\right)^{\frac{1}{2}}}
\end{align}
where $P_T$, $\bar P_T$, $\tilde P_T$ and $Q_{T, 0}$ are the solutions to (\ref{comPT}), (\ref{gramequalforFred}), (\ref{gramequalforFmixed}) and (\ref{gengenlyapobsdet}), respectively.
\end{thm}
\begin{proof}
The result follows from (\ref{absoutbound}) and (\ref{lowerboundy}).
\end{proof}
The bound in Theorem \ref{thmerrorbound} provides a good a priori error estimate, an indicator for the quality of the reduced system. Error bounds for related methods in a deterministic framework can be found in 
\cite{morGugA04, redmannL2T_TLBT, redmannkuerschner}.
\begin{remark}\label{remark_bound}
Notice that if we aim to reduce a Black-Scholes model with output $y_c$ instead of a Heston model with output $y$, the energy $\left\|y_c\right\|_{L^2}$ is explicitly known according to Section \ref{diffobs} given that 
$A, N_1, \ldots, N_q$ commute. Then, it holds that $\left\|y_c\right\|_{L^2}=x_0^T Q_T x_0$ such that we can replace $Q_{T, 0}$ by $Q_{T}$ in Theorem \ref{thmerrorbound}.
\end{remark}

\begin{remark}
  \label{rem:error-payoff}
  For option pricing, we typically consider expectations of $f(y(t))$ for some
  payoff function $f$. If $f$ is Lipschitz (e.g., for put and call options),
  then the error bound of Theorem~\ref{thmerrorbound} immediately carries
  over. However, there are relevant financial options with even discontinuous
  payoffs, for instance digital options. In this case, a general error
  analysis is difficult. However, we would like to point out that financial
  models often have inherent smoothing properties, which allow us to
  effectively mollify the payoff without adding additional bias. We refer to
  \cite{BST18} for an application of this property to option pricing with QMC
  and adaptive sparse grids quadrature methods.
\end{remark}

\section{Numerical experiments}\label{numerics}

We apply the MOR technique motivated in Section \ref{sec:reach} and explained in Section \ref{subsecprocedure}.
The goal is to accurately approximate payoff functions associated with the large asset price model (\ref{stochstatenew}) (these are functions of the quantity of interest in (\ref{output})) 
by payoff functions of the reduced system (\ref{generalreducedsys}). This type of problem is of particular interest if we price European options with an underlying high-dimensional Heston model because 
computational complexity can be reduced. Moreover, since the reduced system shows good pathwise approximations, it can be of interest in the context of Bermudan options because regression based methods \cite{LS2001, TV2001}
 suffer from 
the curse of dimensionality which makes them inaccurate in a large-scale setting. Below, we consider a particular Heston model (\ref{original_system}) and illustrate the quality of the reduction in dependence
of the covariance matrix $\mathbf K$ of the noise process.\smallskip

%We begin with a Black Scholes model, a special case of system (\ref{original_system}), since the effects
%that we observe are the same as in the general case of a randon volatility process. The advantage of starting with a deterministic volatility is that the error bound in Theorem \ref{thmerrorbound} is exact for the example 
%considered below. This makes the error analysis simpler. Subsequently, we illustrate that that the dimension reduction also works for system (\ref{original_system}) with random volatilities (Heston model).

We consider the following linear stochastic differential equation that represents an asset price model:
\begin{align}\label{scholesnum}
             dx_i(t)&=\mathbf{\mathrm{r}} x_i(t)dt+ \xi_i\sqrt{v(t)} x_i(t) dW_i(t),\quad x_i(0)=x_{0, i},
            \end{align}
where $x_i$ denotes the $i$th component of a price process $x$ ($i=1, \ldots, n$). Moreover, we assume that $\mathbf{\mathrm{r}}=0.02$ is the fixed interest rate and $\xi_i\in[0.2, 0.7]$ are volatility parameter
sampled from a uniform distribution.
Now, we can rewrite equation (\ref{scholesnum}) in order to guarantee the form given in (\ref{stochstatenew}). The respective matrices are \begin{align}\label{specualmat}
 A= \mathbf{\mathrm{r}} I,\quad N_i = \xi_i  e_i e_i^T,\quad B=x_0\quad \text{and}\quad z=1,                                                                                                                                                                                             
                                                                                                                                                                                          \end{align}
where $e_i$ is the $i$th unit vector in $\mathbb R^n$, $q=n$, and assuming that we are interested in a single initial value $x_0$ only. In this particular situation, the matrices $A, N_1, \ldots, N_n$ commute 
and are symmetric. The variance process is $v(t)=\min\{\bar v(t), c\}$, $t\in [0, T]$, where $\bar v$ is the solution to the following stochastic differential equation: \begin{align}\label{unboundedvoleq}
        d\bar v(t)=a(b-\bar v(t))dt+\bar\sigma\sqrt{\bar v(t)}d\bar B(t), \quad v(0)=v_0,                
                       \end{align}
where $b=0.2$ is the long run average variance, $a=0.2$ is the rate characterizing the speed of convergence of the average variance and $\bar\sigma=0.15$ is the volatility of the volatility process.
Furthermore, $\bar B$ is a standard Brownian motion with respect to the filtration $(\mathcal F_t)_{t\in [0, T]}$ that negatively correlated with the other standard Brownian motions $W_i$, i.e.,
$\mathbb E[\bar B(t) W_i(t)] = \rho_i t$, where $\rho_i<0$. The parameters $\xi_i, a, b, \bar \sigma$ are chosen to have an average volatility around $0.2$, i.e., \begin{align*}
\mathbb E \left[\sqrt{v(t)}\right] \frac{1}{n} \sum_{i=1}^n \xi_i\approx 0.2.\end{align*}

Notice that in order to fit the theory, the process $v$ is bounded by a constant $c$ since generally $\bar v$ is unbounded. Practically, we simulate a certain number of paths of $\bar v$ and choose 
a $c$ that represents a bound of these simulated paths such that those
coincide with the respective paths of $v$.

Suppose that the quantity of interest is now some one dimensional partial information $y$ of the price process $x$ that is the form \begin{align}\label{specialoutput}
                    y(t)=C x(t),                                                                                                                            
                                                                                                                                               \end{align}
where the output matrix is $C=[1,\; 1, \ldots 1]$. We now determine a reduced system (\ref{generalreducedsys}). To do so, the matrices $P_T$ and $Q_T$ need to be computed in order to conduct the balancing procedure 
described in Section \ref{subsecprocedure}. Fortunately, these matrices can be derived explicitly from (\ref{matrixequalforF}) and (\ref{ODEforG}). Plugging in (\ref{specualmat}) into these equations, we obtain 
 \begin{align*}
\dot F(t)= 2 \mathbf{\mathrm{r}} I F(t)+c\sum_{i, j=1}^n e_i e_i^T F(t) e_j e_j^T \xi_i \xi_j k_{ij}, \quad F(0)= x_0 x_0^T,\\
\dot G(t)= 2 \mathbf{\mathrm{r}} I G(t)+c\sum_{i, j=1}^n e_i e_i^T G(t) e_j e_j^T \xi_i \xi_j k_{ij}, \quad G(0)=C^T C.
\end{align*}
By multiplying the above equations with $e_i^T$ from the left and with $e_j$ from the right, we can see that these equations can be solved component-wise. The entries of $F(t)=(f_{ij})_{i, j=1, \ldots n}$ satisfy
\begin{align*}
 \dot f_{ij}(t)= (2 \mathbf{\mathrm{r}}+c\xi_i \xi_j k_{ij}) f_{ij}(t), \quad f_{ij}(0)= x_{0,i} x_{0, j},
\end{align*}
such that $f_{ij}(t) = \expn^{h_{ij} t}x_{0,i} x_{0, j}$, where $h_{ij}:= (2 \mathbf{\mathrm{r}}+c\xi_i \xi_j k_{ij})$. Integrating $f_{ij}$ over $[0, T]$, we find that $P_T=(p_{ij})_{i, j=1, \ldots n}$ is given by 
\begin{align*}
 p_{ij} = \frac{\expn^{h_{ij} T}-1}{h_{ij}} x_{0,i} x_{0, j}.
\end{align*}
Analogously, it holds that $Q_T=(q_{ij})_{i, j=1, \ldots n}$ is represented by \begin{align*}
 q_{ij} = \frac{\expn^{h_{ij} T}-1}{h_{ij}} e_i^T C^T C e_j.
\end{align*}
Since $P_T$ and $Q_T$ are given explicitly, the reduced system (\ref{generalreducedsys}) with output $\tilde y$ comes basically for free in terms of computational time. This also means that we are able to derive 
a reduced model if the number of assets is very large.
%Moreover, the bound in 
%(\ref{absoutbound}) is exact for this example. We are also in the situation of Remark \ref{remark_bound}, i.e., $Q_{T, 0}$ can be replaced by $Q_T$ in Theorem \ref{thmerrorbound} and we then determine the exact
%relative $L^2$-error with this bound, i.e.,
%\begin{align*}
%\frac{ \left\|y-\tilde y\right\|_{L^2} }{\left\|y\right\|_{L^2}} = \frac{\left(\trace(C P_T C^T)-2 \trace(C \bar P_T C_1^T)+\trace(C_1 \tilde P_T C_1^T)\right)^{\frac{1}{2}}}{\left(x_0^T Q_{T} x_0\right)^{\frac{1}{2}}}
%\end{align*}
%where $\tilde P_T$ and $\bar P_T$ are computed from (\ref{gramequalforFred}) and (\ref{gramequalforFmixed}) with $c=1$. 
We investigate the reduction quality for three different covariance
matrices. First we consider a matrix with both small and large correlations between the noise processes. 
We first choose $\mathbf K=\mathbf K_0:=\tau \tau^T$ according to \cite{doust}, 
where $\tau = [\tau_1,\; \tau_2, \ldots \tau_n]$ has columns $\tau_i$ generated by a vector $s=(s_i)_{i=1, \ldots, n-1}$ of samples $s_i$ of independent uniformly distributed random variables with values in $[0.8, 1]$:
\begin{align*}
\tau_1=\left(\begin{smallmatrix}1\\ \cp(s)\end{smallmatrix}\right), \quad \tau_2=\sqrt{1-s_1^2}\left(\begin{smallmatrix}0\\ 1\\ \cp(s_{2:n-1})\end{smallmatrix}\right),\ldots, 
\tau_n=\sqrt{1-s_{n-1}^2}\left(\begin{smallmatrix}0\\ \vdots\\ 0\\1\end{smallmatrix}\right).
\end{align*}
Above, we set $s_{\ell:n-1}:=(s_\ell, s_{\ell+1},\ldots, s_{n-1})^T$ and \begin{align*}
  \cp(s):= \left[s_1, s_1 s_2, \ldots, s_1 s_2 \dots s_{n-1}\right]^T.
                                                                         \end{align*}
Further, we study the two extreme cases of $\mathbf K=I$ (independent noise processes) and $\mathbf K=\mathbf{1}\mathbf{1}^T$ (perfect correlation),
where $\mathbf{1}$ is an $n$-dimensional vector of ones. We choose $n=100$, $T=1$, $c=0.6$ and the initial conditions $x_{0, i}\in [0, 1.25]$ are generated randomly. 
For $\mathbf K=\mathbf K_0$ we choose $\rho_i\in[-0.9, 0)$, in case of $\mathbf K=I$ we have $\rho_i=-0.09$ and we fix $\rho_i=-0.5$ for $\mathbf K=\mathbf{1}\mathbf{1}^T$. In all the numerical experiments below, 
$2$e$06$ samples are generated.

\subsection{Approximation error in the quantity of interest}\label{numericspathwise}
We begin with analyzing the error between the output of the full system $y$ and the output of the reduced system $\tilde y$. The first question is how to choose the dimension $\tilde n$ of system (\ref{generalreducedsys}).
The HSVs, i.e., $\sigma_i = \sqrt{\lambda_i(P_TQ_T)}$ are a very good indicator for a suitable choice since the smaller $\sigma_i$, the less important the $i$th state component $\hat x_i$ in the balanced 
system (\ref{balancingtransformation}) according to what we have derived in Section \ref{sec:reach}. We can see these values in logarithmic scale for different covariance matrices $\mathbf K$ in Figure \ref{hsvplot}. 
In each case, we observe that one variable dominates the dynamics meaning that no matter how $\mathbf K$ is chosen, a scalar reduced-order model already leads to a relatively good approximation.
Moreover, we see that the reduction is expected to be least efficient if all noise processes are independent due to a slow decay of the HSVs. However, 
with perfect correlation ($\mathbf K=\mathbf{1}\mathbf{1}^T$), we only have four non zero HSVs meaning that the $100$-dimensional model can be perfectly approximated by a system of four variables. In the case of
$\mathbf K=\mathbf K_0$
the performance is in between the independent and perfectly correlated scenarios. This fits to our general observation that the higher the correlation, the better the algorithm works.\smallskip

We conclude this subsection by a discussion on the error between the outputs $y$ and $\tilde y$ of systems (\ref{original_system}) and (\ref{generalreducedsys}) for this particular example. We determine the 
relative $L^2$-error and the corresponding error bound in Theorem  \ref{thmerrorbound}. This bound that is the right-side of (\ref{ebnumerics}) is denoted by $\mathcal{EB}$ here. 
The error bound is relatively tight for the example. In most of the cases it estimates the exact error by factor of three, see Tables \ref{table_computation}, \ref{table_computation2} and \ref{table_computation3}. 
As supposed from the HSVs, very good results are obtained for very small $\tilde n$ if $\mathbf K=\mathbf{1}\mathbf{1}^T$, compare with Figure \ref{errorK0}. However, in the situation of $\mathbf K=\mathbf K_0$,  
a reduced-order between $4$ and $10$ shows a low error, too.
\begin{table}[th]
\centering
\begin{minipage}{.48\linewidth}
\begin{tabular}{|c|c|c|}\hline
$\tilde n$  & $\left\|y-\tilde y\right\|_{L^2} /\left\|y\right\|_{L^2}$ & $\mathcal{EB}$\\
\hline
\hline
$1$ & $5.10$e$-03$& $1.55$e$-02$\\
$4$ & $2.02$e$-03$ & $6.18$e$-03$\\
$10$ & $7.01$e$-04$ & $2.17$e$-03$\\
$15$ & $3.95$e$-04$ & $1.23$e$-03$ \\
$20$ & $2.39$e$-04$ & $7.46$e$-04$ \\
$25$ & $1.62$e$-04$ & $5.06$e$-04$ \\
$50$ & $3.56$e$-05$ & $1.13$e$-04$ \\
\hline
\end{tabular}\caption{Relative $L^2$-error and error bound for $\mathbf K=\mathbf K_0$.}
\label{table_computation}
\end{minipage}
\begin{minipage}{.48\linewidth}
\begin{tabular}{|c|c|c|}\hline
$\tilde n$  & $\left\|y-\tilde y\right\|_{L^2} /\left\|y\right\|_{L^2}$ & $\mathcal{EB}$\\
\hline
\hline
$1$ & $2.56$e$-03$& $7.79$e$-03$\\
$4$ &  $2.23$e$-03$ & $6.79$e$-03$\\
$10$ & $1.69$e$-03$ & $5.12$e$-03$\\
$15$ & $1.46$e$-03$ & $4.42$e$-03$ \\
$20$ & $1.27$e$-03$ & $3.84$e$-03$ \\
$25$ & $1.08$e$-03$ & $3.28$e$-03$ \\
$50$ & $5.01$e$-04$ & $1.51$e$-03$ \\
\hline
\end{tabular}\caption{Relative $L^2$-error and error bound for $\mathbf K=I$.}
\label{table_computation2}
\end{minipage}
\begin{minipage}{.48\linewidth}
\begin{tabular}{|c|c|c|}\hline
$\tilde n$  & $\left\|y-\tilde y\right\|_{L^2} /\left\|y\right\|_{L^2}$ & $\mathcal{EB}$\\
\hline
\hline
$1$ & $2.40$e$-03$& $5.48$e$-03$\\
$2$ & $3.25$e$-06$ & $1.34$e$-05$\\
$3$ &$2.94$e$-09$  & $2.98$e$-08$\\
$4$ &$3.76$e$-12$  & $2.76$e$-09$ \\
\hline
\end{tabular}\caption{Relative $L^2$-error and error bound for $\mathbf K=\mathbf{1}\mathbf{1}^T$.}
\label{table_computation3}
\end{minipage}
\end{table}

\begin{figure}[ht]
\begin{minipage}{0.48\linewidth}
 \hspace{-0.5cm}
\includegraphics[width=1.1\textwidth,height=6cm]{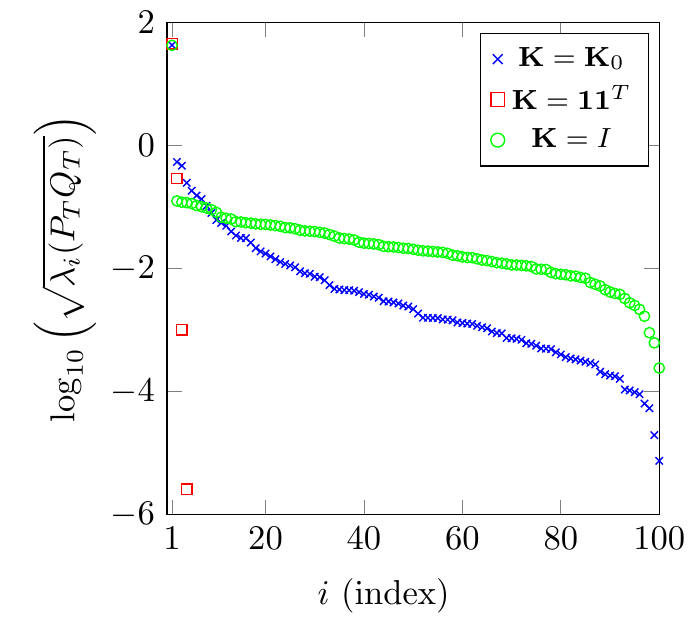}
\caption{Logarithmic HSVs of the large-scale asset model.}\label{hsvplot}
\end{minipage}\hspace{0.5cm}
\begin{minipage}{0.48\linewidth}
 \hspace{-0.5cm}
\includegraphics[width=1.1\textwidth,height=6cm]{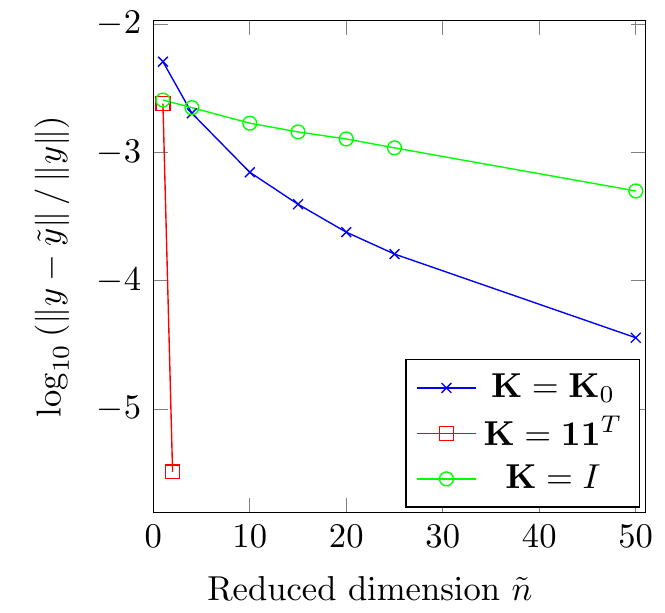}
\caption{Logarithmic relative $L^2$-error between $y$ and $\tilde y$ for $\tilde n\in\{1, 4, 10, 15, 20, 25, 50\}$ if $\mathbf K=\mathbf K_0, I$ and for $\tilde n=1, 2$ if $\mathbf K=\mathbf{1}\mathbf{1}^T$.}\label{errorK0}
\end{minipage}
\end{figure}

We conclude this subsection by briefly discussing the robustness of our algorithm in the parameter $\bar \sigma$ which is the volatility of the volatility processes $\bar v$ given in \eqref{unboundedvoleq}
 and the bound $c$ that that defines the truncated volatility process $v$. We consider the 
same framework as above apart from an enlarged volatility parameter which we choose $\bar \sigma = 0.25$ for the moment. As mentioned above, we practically simulate and use paths of $\bar v$ and subsequently fix 
the constant $c$ representing a bound for each of the simulated paths. Theoretically, we then work with $v$ since we need its boundedness to motivate the algorithm and in order to derive the error bound in Theorem 
\ref{thmerrorbound}. The new choice of $\bar \sigma$ of course leads to a larger $c$. According to the simulated paths, $c\geq 1.1$ is suitable. Picking $c=1.1$, we can see that the relative error in Table \ref{table_robust}
is basically the same as in Table \ref{table_computation} such that a higher volatility parameter does not really affect the relative error. Moreover, choosing $c=2$ gives us an error that is only slightly 
worse such that we observe a relatively good robustness in $c$. Interestingly, we can also select $c=0.6$ as a parameter in the Gramians $P_T$ and $Q_T$ and run the model reduction procedure based on these non admissible 
Gramians and even get a small gain. It might be because the majority of the paths of $\bar v$ stay below the bound of $c=0.6$. Therefore, most of the dynamics are still captured in the reduced system. 
Although we have seen a relatively robust scheme in the parameter $c$, we suggest to not choose it too small or large in practice.
\begin{table}[th]
\centering
\begin{tabular}{|c|c|c|c|}\hline
& \multicolumn{3}{c}{$\left\|y-\tilde y\right\|_{L^2} /\left\|y\right\|_{L^2}$} \vline \\
\hline 
$\tilde n$  & $c= 0.6$ & $c= 1.1$& $c=2$\\
\hline
\hline
$1$ & $5.23$e$-03$&$5.27$e$-03$ &$5.47$e$-03$\\   
$4$ &$2.06$e$-03$  &$2.09$e$-03$ &$2.22$e$-03$ \\
$10$ &$7.16$e$-04$  &$7.28$e$-04$ &$7.68$e$-04$\\
\hline
\end{tabular}\caption{Relative $L^2$-error for an enlarged volatility parameter $\bar \sigma = 0.25$ with $\mathbf K=\mathbf K_0$ using $c = 0.6, 1.1, 2$.}
\label{table_robust}
\end{table}

%\begin{figure}[ht]
%\begin{minipage}{0.48\linewidth}
% \hspace{-0.5cm}
%\includegraphics[width=1.1\textwidth,height=6cm]{HSVs.pdf}
%\caption{Relative $L^2$-error between $y$ and $\tilde y$ and error bound for $\tilde n=1, 2, 3, 4$ and $\mathbf K=\mathbf{1}\mathbf{1}^T$.}\label{errorKones}
%\end{minipage}\hspace{0.5cm}
%\begin{minipage}{0.48\linewidth}
% \hspace{-0.5cm}
%\includegraphics[width=1.1\textwidth,height=6cm]{HSVs.pdf}
%\caption{Relative $L^2$-error between $y$ and $\tilde y$ and error bound for $\tilde n=1, 4, 10, 15, 20, 25, 50$ and $\mathbf K=I$.}\label{errorKId}
%\end{minipage}
%\end{figure}

% The components of the noise process $W$ are generated as follows:\begin{align*}
 %  W_i(t)= \rho \bar B(t)+\sqrt{1-\rho^2}\bar W_i(t)                                                                                                                                                                               
  %                                                                                                                                                                                \end{align*}
%with $\rho\in (0, 1)$ and with $\bar B, \bar W_1, \ldots, \bar W_n$ being independent standard Brownian motions. For that reason, the covariance matric of $W$ is $\mathbf K=(k_{ij})_{i, j=1, \ldots, n}$, where 
%$k_{ij}= 1$ if $i=j$ and $k_{ij}= \rho^2$ else.

\subsection{Approximation error in the payoff}\label{weakapprox}

We have seen in Section \ref{numericspathwise} that the quantity of interest $y$ of system (\ref{original_system}), that we specified in (\ref{specialoutput}), can be well approximated by the output 
$\tilde y$ of the reduced system (\ref{generalreducedsys}) in a path-wise sense on some interval $[0, T]$. However, it is often of interest to consider weak errors instead. Therefore, we consider the following 
payoff function \begin{align*}
f(y) = \max\left\{y-K, 0\right\}                 
                \end{align*}
that plays a role in the context of European call options, where $K$ denotes the strike price. We compare the expected payoff $\mathbb E f\left(y(T)\right)$ at time $T$ with the one associated with 
the reduced system, which is $\mathbb E f\left(\tilde y(T)\right)$, for $K= \langle\mathbf 1, x_0\rangle_2$ in Figure \ref{payoffplot}. The relative errors in the expected payoff are larger if the
correlations between the noise 
processes are small. The approximation works best if $\mathbf K=\mathbf{1}\mathbf{1}^T$. As displayed in Figure \ref{payoffplot}, the error is around $5$e$-06$ for $\tilde n=2$ and our simulations also show an error of $5$e$-13$ already for $\tilde n=4$. 
If smaller correlations are involved, a larger $\tilde n$ needs to be chosen. However, selecting $4 \leq \tilde n \leq 10$ for $\mathbf K=\mathbf K_0$ already leads to a good estimate of the original payoff.  \smallskip

Looking at Table \ref{table1} we observe that weak error (error in the expected payoff) is of the same or of smaller error than the strong error in Table \ref{table_computation}. Moreover, it can be seen that 
the approximation in the payoff is better if the strike price is below the value of the basket $\langle\mathbf 1, x_0\rangle_2$ at time zero and it is worse if the strike price is above 
$\langle\mathbf 1, x_0\rangle_2$.\smallskip

So far, the weak error has not yet been analyzed concerning error bounds etc. We believe that it requires advanced techniques to succeed in this direction.
\begin{table}[th]
\centering
\begin{tabular}{|c|c|c|c|}\hline
& \multicolumn{3}{c}{$\left\vert \mathbb E f(y(T))-\mathbb E f(\tilde y(T))\right\vert /\left\vert \mathbb E f(y(T))\right\vert$} \vline \\
\hline 
$\tilde n$  & $K= 0.9\; \langle\mathbf 1, x_0\rangle_2$ & $K= \langle\mathbf 1, x_0\rangle_2$& $K= 1.1\; \langle \mathbf 1, x_0\rangle_2$\\
\hline
\hline
$1$ & $9.54$e$-04$&$1.10$e$-03$ &$1.98$e$-02$\\   
$4$ &$8.81$e$-04$  &$6.92$e$-04$ &$5.28$e$-03$ \\
$10$ &$2.78$e$-04$  &$4.13$e$-04$ &$8.80$e$-04$\\
$15$ &$3.41$e$-05$ & $8.63$e$-05$& $2.22$e$-04$\\
$20$ &$8.62$e$-06$ &$1.79$e$-05$ &$2.25$e$-05$\\
$25$ &$2.56$e$-06$ &$7.22$e$-06$ &$2.19$e$-05$\\
$50$ &$1.12$e$-07$ &$2.22$e$-07$ &$7.09$e$-07$ \\
\hline
\end{tabular}\caption{Relative error in the payoff function for $\mathbf K=\mathbf K_0$ and different strike prices $K$.}
\label{table1}
\end{table}
\begin{figure}[ht]
\centering
\includegraphics[width=7.5cm,height=7.5cm]{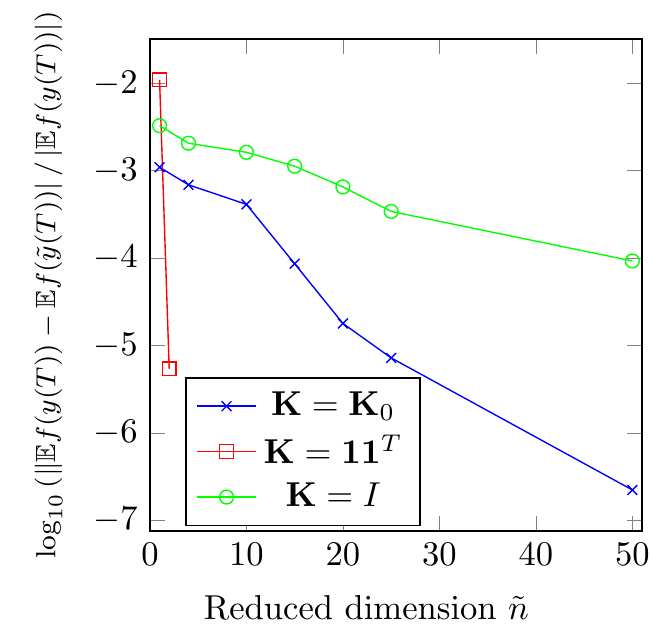}
\caption{Logarithmic relative error expected payoff for $K= \langle\mathbf 1, x_0\rangle_2$ and different reduced dimensions.}\label{payoffplot}
\end{figure}
% \begin{remark}
% We have seen a good path-wise performance of our method in Section \ref{numericspathwise} and an even better approximation in the payoff in this section. Therefore, we see the potential of pricing 
% high-dimensional Bermudan options with the help of MOR. Classical regression based schemes as in \cite{LS2001, TV2001} cannot accurately determine values of Bermudan options with high-dimensional 
% underlying asset models due to the curse of dimensionality. However, reducing the asset price model in its dimension and subsequently applying the methods in  \cite{LS2001, TV2001} can be promising. 
% In this context, the potential of our method is briefly illustrated in the next section.
% \end{remark}

\subsection{Applications to Bermudan options}
\label{sec:appl-berm-opti}

We have seen a good path-wise performance of our method in Section \ref{numericspathwise} and an even better approximation in the payoff in this section. Therefore, we see the potential of pricing 
high-dimensional Bermudan options with the help of MOR. Classical regression based schemes as in \cite{LS2001, TV2001} cannot accurately determine values of Bermudan options with high-dimensional 
underlying asset models due to the curse of dimensionality. However, reducing the asset price model in its dimension and subsequently applying the methods in  \cite{LS2001, TV2001} can be promising. 
%In this context, the potential of our method is briefly illustrated in the next section.
%We now apply our techniques to Bermudan or, more generally, American option.
To fix notation, consider an optimal stopping problem with possible
exercise date $\mathcal{J} \subset [0,T]$ with value
\begin{equation}
  \label{eq:optimal-stopping-value}
  u \coloneqq \sup_{\tau \in \mathcal{T}} \mathbb{E}[f(y(\tau))],
\end{equation}
where $\mathcal{T}$ denotes the set of all stopping times w.r.t.~the
filtration $(\mathcal{F}_t)_{t \in \mathcal{J}}$ generated by the Brownian
motions $\overline{B}, W_1, \ldots, W_q$ restricted to $\mathcal{J}$. (Here,
$\mathcal{J}$ is finite in the case of a \emph{Bermudan option}, while
$\mathcal{J} = [0,T]$ in the case of an \emph{American option}. In all cases,
we assume that $0 \in \mathcal{J}$.) Let us further consider the same problem
in the reduced model,
\begin{equation}
  \label{eq:optimal-stopping-reduced}
  \tilde{u} \coloneqq \sup_{\tau \in \mathcal{T}} \mathbb{E}[f(\tilde{y}(\tau))].
\end{equation}
Note that we choose the same class of stopping times for the full and for the
reduced model, guaranteeing that the optimal stopping time for either choice
of the model will be a sub-optimal stopping time for the other choice of
model.
%  Finally,
% let us denote by $\widehat{u}$ a numerical approximation of $\tilde{u}$ --
                                                                     %                              
% \begin{assumption}
%   \label{ass:trivial_enlargement}
%   There are independent Brownian motions $Z_1, \ldots, Z_q$ obtained by rotation of $W_1,
%   \ldots, W_q$ such that the filtration $(\mathcal{F}_t)_{t \in [0,T]}$
%   is generated by $v$ together with $Z_1, \ldots, Z_q$, whereas the filtration
%   $(\mathcal{G}_t)_{t \in [0,T]}$ is generated by $v$ together with $Z_1,
%   \ldots, Z_p$ for some $1 \le p \le q$.
% \end{assumption}

\begin{remark}
  \label{rem:randomized-stopping}
  In contrast to~\eqref{eq:optimal-stopping-value} and
  \eqref{eq:optimal-stopping-reduced}, one could also consider restricting the
  class of admissible stopping times to stopping times w.r.t.~the filtration
  $\mathbb{G} \coloneqq (\mathcal{G}_t)_{t \in \mathcal{J}}$ generated by $v$
  and $\tilde{x}$ for the reduced model and stopping times w.r.t.~the
  filtration $\mathbb{H} \coloneqq (\mathcal{H}_t)_{t \in \mathcal{J}}$
  generated by $v$ and $x$ for the full model, respectively. In general, this
  means that optimal stopping times for either model will not even be
  candidate stopping times for the other model. Often -- but not always --
  $\mathbb{H}$ will be an enlargement of $\mathbb{G}$, and under some
  conditions, this would allow us to understand stopping times for the full
  model as \emph{randomized stopping times} for the reduced model in the sense
  of \cite{gyongy2008randomized}. This implies that the optimal stopping time
  for the full model is a sub-optimal randomized stopping time for the reduced
  model, and its expected payoff is less or equal to the expected payoff of
  the optimal stopping time of the reduced model, w.r.t.~the smaller
  filtration. In general, however, the analysis of the behavior of American
  options for general model reductions within their own generated filtrations
  is beyond the scope of this paper.
\end{remark}
\begin{lem}\label{lem:reduced-american}
  Assume that the reduced model is close to the full model in the sense that
  \begin{equation}\label{eq:reduced-pathwise-L1}
    \mathbb{E}\left[ \sup_{t \in \mathcal{J}} \abs{f(y(t)) - f(\tilde{y}(t))}
    \right] \le \epsilon.
  \end{equation}
  We then have that
  \begin{equation*}
    \abs{u - \tilde{u}} \le \epsilon.
  \end{equation*}
\end{lem}
\begin{proof}
  Let $\tau^\ast$ denote the optimal stopping time for the full problem, i.e.,
  $u = \mathbb{E}[f(y(\tau^\ast))]$. Similarly, let $\tilde{\tau}^\ast$
  denote the optimal stopping time for the reduced problem. We have
  \begin{align*}
    u - \tilde{u} &= \mathbb{E}\left[ f(y(\tau^\ast)) \right] -
                    \mathbb{E}\left[ f\left( \tilde{y}\left( \tilde{\tau}^\ast
                    \right) \right) \right] \\
    &= \underbrace{\mathbb{E}\left[ f(y(\tau^\ast)) \right] - \mathbb{E}\left[ f\left(
      y\left( \tilde{\tau}^\ast \right) \right) \right]}_{\ge 0} + \underbrace{\mathbb{E}\left[ f\left( y\left( \tilde{\tau}^\ast \right)
      \right) \right] - \mathbb{E}\left[ f\left( \tilde{y}\left( \tilde{\tau}^\ast
      \right) \right) \right]}_{\ge - \epsilon}\\
    &\ge - \epsilon,
  \end{align*}
  since $\tilde{\tau}^\ast$ is a sub-optimal stopping time for $f(y)$.

  For the upper bound, we note that $\tau^\ast$ is a
  sub-optimal stopping time for $f(\tilde{y})$
  % randomized stopping time for $f(\tilde{y})$ in the sense
  % of~\cite{gyongy2008randomized}. Hence, the expected payoff from
  % $f(\tilde{y}(\tau^\ast))$ is less or equal to the expected optimal payoff,
  % even when restricted to proper stopping times w.r.t.~$(\mathcal{G}_t)_{t
  %   \in \mathcal{J}}$.
  % \emph{essentially}
  % a sub-optimal stopping time for $f(\tilde{y})$. Strictly speaking, as
  % $\tau^\ast$ is a stopping time w.r.t.~the larger filtration
  % $(\mathcal{F}_t)_{t \in \mathcal{J}}$, it does not even qualify as a
  % candidate stopping time for the reduced optimal stopping time. However,
  % since the dynamics of the reduced process $\tilde{y}$ does not depend on the
  % additional information available to $\tau^\ast$, the stopping time cannot
  % take advantage of the additional information and, indeed, is generally
  % suboptimal. 
  % Therefore, we
  and obtain
  \begin{align*}
    u - \tilde{u} &= \mathbb{E}\left[ f(y(\tau^\ast)) \right] -
                    \mathbb{E}\left[ f\left( \tilde{y}\left( \tilde{\tau}^\ast
                    \right) \right) \right] \\
    &= \underbrace{\mathbb{E}\left[ f(y(\tau^\ast)) \right] - \mathbb{E}\left[ f\left(
      \tilde{y}\left( \tau^\ast \right) \right) \right]}_{\le \epsilon} + \underbrace{\mathbb{E}\left[ f\left( \tilde{y}\left( \tau^\ast \right)
      \right) \right] - \mathbb{E}\left[ f\left( \tilde{y}\left( \tilde{\tau}^\ast
      \right) \right) \right]}_{\le 0}\\
    &\le \epsilon,
  \end{align*}
 since $\tilde{\tau}^\ast$ is optimal for $f(\tilde y)$.
\end{proof}

\begin{remark}
Unfortunately, the $L^2$-error bound obtained in Theorem \ref{thmerrorbound} is not a candidate for $\epsilon$ in \eqref{eq:reduced-pathwise-L1} since we cannot bound the $L^\infty$-error by the $L^2$-error. 
However, we can usually estimate $\mathbb{E}\left[ \sup_{t \in \mathcal{J}} \abs{f(y(t)) - f(\tilde{y}(t))}\right]$ empirically by sampling $y$ and $\tilde y$.
Indeed, for numerical purposes, we may always assume that $\mathcal{J}$ is finite, such that the $L^\infty$-norm for a given trajectory is easily computable.
Additionally, we note that estimating $\mathbb{E}\left[ \sup_{t \in \mathcal{J}} \abs{f(y(t)) - f(\tilde{y}(t))}\right]$ is, of course, much easier than estimating $\abs{u - \tilde{u}}$, as no optimal stopping problem needs to be solved.
This type of approximation for the bound between the option prices $u$ and $\tilde u$ is also used in the numerical example below.
Theoretical error bounds of the form \eqref{eq:reduced-pathwise-L1} would require very different techniques than applied in the proof of Theorem \ref{thmerrorbound} and are left for further research.    
\end{remark}

We now create an example, in which the error in \eqref{eq:reduced-pathwise-L1} is already relatively low for a reduced dimension $\tilde n = 5$, since this gives us the certainty that the value of the Bermudan option $\tilde u$ in 
the reduced model is close to the actual value $u$ by Lemma \ref{lem:reduced-american}. To do so, we modify equation \eqref{scholesnum} by choosing $n=30$, a constant volatility and dividends
resulting in the following Black-Scholes model: \begin{align}\label{scholesnum2}
             dx_i(t) =(\mathbf{\mathrm{r}} - \delta) x_i(t)dt+ \xi x_i(t) dW_i(t),\quad x_i(0)=x_{0, i},\quad t\in [0, T],
            \end{align}
where $\mathbf{\mathrm{r}} = 0.02$, $\delta = 0.07$, $\xi = 0.2$, $T=1$ and $i= 1, \ldots, n$. The noise processes are correlated. Their covariance matrix is generated the same way as $\mathbf K_0$ above and we choose randomly generated initial values 
$x_{0, i}\in [0.15, 2.5]$. The quantity of interest is chosen as before, i.e., \begin{align*}
                                                                    y(t) = \sum_{i=1}^n x_i(t)
                                                                               \end{align*}
and the following discounted payoff function is considered \begin{align*}
f(y(t)) = \expn^{-\mathbf{\mathrm{r}} t} \max\left\{y(t)-K, 0\right\}                 
                \end{align*}
with $K = \langle\mathbf 1, x_0\rangle_2$. We have five exercise dates of the associated Bermudan option, i.e., $\mathcal{J} = \left\{0, 0.25, 0.5, 0.75, 1\right\}$.

We apply the MOR technique described in Section \ref{subsecprocedure} in order to obtain the reduced system \eqref{generalreducedsys} with dimensions $\tilde n = 1, \ldots, 5$, Subsequently, we determine 
the error in \eqref{eq:reduced-pathwise-L1}. These errors are stated in the third column in Table \ref{table_bermudan} showing that we are already very close to the actual value $u$ using a reduced system 
with dimension $\tilde n = 5$. In order to determine the fair price for the Bermudan option in the reduced model ($\tilde n = 1, \ldots, 5$), we apply the algorithm of Longstaff and Schwartz \cite{LS2001}. Within this 
regression approach Hermite polynomials of absolute order up to $5$ are used as a basis. Moreover, the payoff function is included in the basis as well. This leads to the values $\tilde u$ in the second column of 
Table \ref{table_bermudan}. The gain in $\tilde u$ from $\tilde n = 1$ to $\tilde n=5$ is relatively low. It is hard to distinguish between both values knowing that the standard deviation of this estimation is 
$0.00105$ which is the same order as the gain. Observing the values $\tilde u$ in Table \ref{table_bermudan} for different reduced order dimensions, it seems that the corresponding bound is not very tight such that 
we suppose that the deviation between $u$ and $\tilde u$ is much lower than $10^{-2}$ when choosing $\tilde n = 5$.
\begin{table}[th]
\begin{tabular}{|c|c|c|}\hline
$\tilde n$  & Value $\tilde u$ of Bermudan option reduced system & $\mathbb{E}\left[ \sup_{t \in \mathcal{J}} \abs{f(y(t)) - f(\tilde{y}(t))}\right]$\\
\hline
\hline
$1$ & $0.99173$& $0.079619$\\
$2$ & $0.99260$ & $0.049844$\\
$3$ & $0.99313$ & $0.025734$\\
$4$ & $0.99295$ & $0.016036$ \\
$5$ & $0.99354$ & $0.012181$ \\
\hline
\end{tabular}\caption{Value Bermudan option based on the reduced asset price model corresponding to \eqref{scholesnum2} and associated error bounds from Lemma \ref{lem:reduced-american}.}
\label{table_bermudan}
\end{table}
Moreover, notice that the value of the European option in the above asset price model is $0.87980$, such that there is a significant difference between both option prices. This makes our method beneficial since computing 
the Bermudan option price in the reduced model leads to a large gain in comparison to the European option price in the full model which would be a good estimator if the values of both types of options do not deviate too much.      
            
\section{Conclusions and outlook}
\label{sec:conclusions}

In this paper, we have shown that model order reduction (MOR) can be an effective technique to construct lower-dimensional surrogate models of large-scale financial models.
These surrogate models can be tackled by higher-order computational methods than Monte Carlo simulation. We
construct a specific path-wise MOR method, and test it in a multi-dimensional
Heston model. The MOR turns out to work very well for European basket option
pricing, especially when the individual assets are strongly correlated (a very
realistic scenario). For instance, in a common Doust-type correlation regime,
for $n = 100$ assets, a reduced model with dimension $\tilde n = 1$ was able to
capture the price of an ITM basket option up to a relative error of $10^{-3}$,
whereas for an OTM option we obtained the same error bound with $\tilde n = 10$, which is
still a very significant dimension reduction.

Of course, this paper only scratches the surface of applications of MOR in
finance. In particular, we identify two very relevant extensions that will be
highly beneficial in a financial context. On the one hand, consider that we
have restricted ourselves to linear dynamics and essentially linear payoff
functions -- in the sense that the payoff is assumed to be a non-linear
function of a low-dimensional projection of the full price process. Both
restrictions can be quite relevant in finance. Allowing non-linear dynamics
opens up the possibility of including the stochastic variance process in the
model order reduction, as well as having local volatility
components. Techniques for MOR in non-linear dynamics have already been
developed in the deterministic case \cite{morBenG17, morBenG19, KramerWilcox},
and have been extended to stochastic differential equations in some special
cases \cite{redstochbil}. General non-linear payoff functions are also
relevant in finance, think of max-call options. One strategy already available
in our framework is to choose $C$ to be the identity matrix. 

On the other hand, note that the MOR framework developed in this paper is
strong in the probabilistic sense, i.e., we try to approximate the process
itself. In many financial application, we are interested in
weak approximations, i.e., we want to approximate the distribution of the
process. As this is a much weaker concept, even better MOR techniques are
conceivable. However, developing an appropriate framework does not seem
obvious, and it is unclear how to proceed in this direction.
%\todo[inline,caption={}]{
%Idea: Do MOR for the Kolmogorov PDE, such that the reduced equation is still the
%Kolmogorov equation of some Markov process.
%}

\section*{Acknowledgments}
The authors would like to thank the anonymous reviewers for their helpful, constructive and detailed comments that greatly contributed to improving the paper. 

\appendix

\section{Resolvent positive operators}\label{appendixbla}

Let $\left(H^n, \langle \cdot, \cdot\rangle_F\right)$ be the Hilbert space of symmetric $n\times n$ matrices,
where $\langle M_1, M_2\rangle_F:= \trace(M_1^T M_2)$ is the Frobenius inner product of two matrices $M_1$ and $M_2$. The corresponding norm is defined by $\left\|M_1\right\|_F^2:= \langle M_1, M_1\rangle_F$.
Moreover, let $H^n_+$ be the subset of symmetric positive semidefinite matrices. We now define positive and resolvent positive operators on $H^n$.
\begin{defn}
A linear operator $L: H^n\rightarrow H^n$ is called positive if $L(H^n_+)\subset H^n_+$. It is resolvent positive
if there is an $\alpha_0\in\mathbb R$ such that for all $\alpha>\alpha_0$ the operator $(\alpha I - L)^{-1}$ is positive.
\end{defn}\\
The operator $\mathcal L(X):=AX + X A^T$ is resolvent positive for $A\in\mathbb R^{n\times n}$ which is, e.g., shown in \cite{damm}.
Moreover, $\Pi(X):=c \sum_{i, j=1}^q N_i X N_j^T k_{ij}$ is positive for $N_i\in\mathbb R^{n\times n}$ by \cite[Proposition 5.3]{redmannspa2}. This implies that the generalized Lyapunov operator 
$\mathcal L + \Pi$ is resolvent positive. We now state an equivalent characterization for resolvent positive operators in the following. It can be found in a
more general form in \cite{damm, elsner, schneidervid}.
\begin{thm}\label{equiresolpos}
A linear operator $L: H^n\rightarrow H^n$ is resolvent positive if and only if $\langle V_1, V_2\rangle_F=0$ implies $\langle L V_1, V_2\rangle_F\geq 0$ for $V_1, V_2\in H^n_+$.
\end{thm}

\section{Pending proofs}
We prove Lemmas \ref{lemforstab} and \ref{lem3} in the following two subsections.

\subsection{Proof of Lemma \ref{lemforstab}}\label{secprooflem2}

We apply Ito's product rule to $x(t) x^T(t)$ and obtain
\begin{align*}
 d\left(x(t) x^T(t)\right)= dx(t) x^T(t) + x(t) dx^T(t) + dx(t) dx^T(t).
\end{align*}
Inserting (\ref{stochstatenew}) yields \begin{align}\label{usualprod}
  dx(t) x^T(t) + x(t) dx^T(t) = &Ax(t)x^T(t)dt+\sum_{i=1}^q \sqrt{v(t)}N_i x(t)x^T(t) dW_i(t) \\ \nonumber
  &+ x(t) x^T(t) A^T dt+\sum_{i=1}^q \sqrt{v(t)}x(t) x^T(t)N_i^T dW_i(t).
\end{align}
With (\ref{stochstatenew}) and using that $dW_i(t) dW_j(t) = k_{ij} dt$, we find \begin{align}\label{quadvarterm}
 dx(t) dx^T(t) = v(t)\sum_{i, j=1}^q N_i x(t) x^T(t) N_j^T k_{ij} dt.
\end{align}
Let $e_i$ denote the $i$th unit vector. Then, we have
\begin{align*}
  k_{ij}=e_i^T \mathbf K^{\frac{1}{2}} \mathbf K^{\frac{1}{2}} e_j= \sum_{k=1}^q \langle \mathbf K^{\frac{1}{2}} e_i, e_k\rangle_2 \langle \mathbf K^{\frac{1}{2}}e_j , e_k\rangle_2.
  \end{align*}
Using this fact, we obtain that \begin{align*}
\sum_{i, j=1}^q N_i x(t) x^T(t) N_j^T k_{ij} = \sum_{k=1}^q \left(\sum_{i=1}^q N_i x(t) \langle \mathbf K^{\frac{1}{2}} e_i, e_k\rangle_2\right) 
\left(\sum_{j=1}^q N_j x(t)\langle \mathbf K^{\frac{1}{2}} e_j, e_k\rangle_2\right)^T \geq 0
                                \end{align*}
is a positive semidefinite matrix. Hence, we can enlarge the right-side of (\ref{quadvarterm}) by replacing $v$ by its bound $c$. This leads to 
\begin{align}\label{ineqquadvarterm}
 dx(t) dx^T(t) \leq c \sum_{i, j=1}^q N_i x(t) x^T(t) N_j^T k_{ij} dt.
\end{align}
We apply the expected value to both sides of (\ref{usualprod}) and (\ref{ineqquadvarterm}). Since the Ito integrals have mean zero, we have
\begin{align*}
 \frac{d}{dt}\mathbb E \left[x(t) x^T(t)\right]\leq A\mathbb E \left[x(t)x^T(t)\right]+ \mathbb E \left[x(t)x^T(t)\right]A^T + c\sum_{i, j=1}^q N_i \mathbb E \left[x(t) x^T(t)\right] N_j^T k_{ij},
\end{align*}
which concludes the proof.

\subsection{Proof of Lemma \ref{lem3}}\label{secprooflem3}

We combine (\ref{matrix_ineq}) with (\ref{matrix_eq}) and obtain 
\begin{align*}
\dot Y(t) \geq L(Y(t)),
\end{align*}
where $Y:= Z-X$.
We define the difference function $D(t) := \dot Y(t) - L(Y(t)) \geq 0$ and consider the following perturbed differential equation 
\begin{align*}
\dot Y_\epsilon (t) = L(Y_\epsilon(t)) + D(t) + \epsilon I
\end{align*}
with parameter $\epsilon\geq 0$ and initial state $Y_\epsilon (0)= Y(0) + \epsilon I$. We see that $Y_0(t)= Y(t)$ for all $t\in [0, T]$ since $Y_0 - Y$ solves (\ref{matrix_eq})  with initial condition zero.
Since $Y_\epsilon$ continuously depends on $\epsilon$ and the initial data, we have $\lim_{\epsilon \rightarrow 0} Y_\epsilon(t) = Y_0(t) = Y(t)$ for all $t\in [0, T]$.
\smallskip

We want to prove that $Y_\epsilon(t)$ is positive definite for all $t$ and all $\epsilon>0$. To do so, let us assume the converse, i.e., 
there is a $\tilde u\ne 0$ and a $\tilde t>0$ such that $\tilde u^T Y_\epsilon(\tilde t) \tilde u \leq 0$.
We know that $f_\epsilon(u, t):= u^T Y_\epsilon(t)  u$ is positive at $t=0$ for all $u\in\mathbb R^n\setminus \{0\}$ since $Y(0)\geq0$ by assumption.
Since $f_\epsilon$ is non-positive in some point $(\tilde u, \tilde t)$ and due to the continuity of $t\mapsto Y_\epsilon(t)$, there is a point $t_0\in (0, \tilde t]$ 
for which \begin{align}\label{contradicttionprop}
u_0^T Y_\epsilon(t_0) u_0 = 0 \quad \text{and} \quad u_0^T Y_\epsilon(t) u_0 > 0,\quad t<t_0,
\end{align}
for some $u_0\ne 0$, whereas $u^T Y_\epsilon(t_0) u\geq 0$ for all other $u\in\mathbb R^n$. Since $L$ is resolvent positive, $0 = u_0^T Y_\epsilon(t_0) u_0 = \langle Y_\epsilon(t_0), u_0u_0^T\rangle_F$ implies
$0\leq \langle L(Y_\epsilon(t_0)), u_0u_0^T\rangle_F = u_0^T L(Y_\epsilon(t_0))u_0$ by Theorem \ref{equiresolpos}. Hence, we have 
\begin{align*}
\left. u_0^T \dot Y_\epsilon(t) u_0\right\vert_{t=t_0} = u_0^T L(Y_\epsilon(t_0)) u_0 +u_0^T D(t_0) u_0 + \epsilon \left\|u_0\right\|_2^2 > 0.
\end{align*}
Consequently, we know that there are $t<t_0$ close to $t_0$ for which $u_0^T Y_\epsilon(t) u_0<0$. This contradicts (\ref{contradicttionprop}) and hence our assumption is wrong such that 
$Y_\epsilon(t)$ is positive definite for all $t\in [0, T]$ and $\epsilon>0$. Taking the limit of $\epsilon \rightarrow 0$, we obtain $Y(t)\geq 0$ for all $t\in [0, T]$ which concludes the proof.

\bibliographystyle{plain}
%\bibliography{refererences}

\end{document}